%% file: main.tex
\newif\ifjgcd
\jgcdfalse 

\ifjgcd
    \documentclass[10pt,journal]{new-aiaa-custom} 
\else
    \documentclass[11pt,a4paper]{article} 
\fi

\input{preamble/packages}
\input{preamble/commands}
\input{preamble/switch}
\input{preamble/settings}
\input{preamble/setup}

\begin{document}


\input{preamble/titlepage}

\input{sections/introduction}
\input{sections/preliminaries}
\input{sections/free_final_time}
\input{sections/robust}
\input{sections/resilient}

\input{sections/pdg}
\input{sections/robustification}
\input{sections/resilience}

\input{sections/conclusions}

\newpage

\input{preamble/bibliography}

\end{document}

%% file: preamble/packages.tex
\usepackage{etoolbox}
\makeatletter
\robustify\@latex@warning@no@line 
\makeatother

\usepackage[left=0.75in, right=0.75in, top=0.75in, bottom=0.75in]{geometry}
\usepackage{lmodern}
\usepackage[utf8]{inputenc}
\usepackage[T1]{fontenc}
\usepackage[dvipsnames]{xcolor}
\usepackage{amsmath}
\usepackage{amssymb}
\usepackage{physics} 
\usepackage{mathtools} 
\usepackage{longtable,tabularx}
\usepackage[nocomma]{optidef}
\usepackage{algorithm}
\usepackage{algpseudocode}
\usepackage{nicematrix}
\usepackage{arydshln}
\usepackage{hhline}
\usepackage{cancel}
\usepackage{relsize}
\usepackage{tcolorbox}
\usepackage[toc]{appendix}
\usepackage{amsthm}
\usepackage{enumitem}
\usepackage{setspace}
\usepackage{cancel}
\usepackage{lastpage}
\usepackage{fancyhdr}
\usepackage{lipsum}
\usepackage{optidef}
\usepackage{accents}
\usepackage{parskip}

\usepackage{booktabs}
\usepackage{tcolorbox}
\tcbuselibrary{skins}

\usepackage[blocks]{authblk}
\usepackage{diagbox}

\usepackage{makecell}

%% file: preamble/commands.tex
\newcommand{\behcet}{Beh{\c{c}}et A{\c{c}}{\i}kme{\c{s}}e}

\definecolor{darts}{HTML}{009670}
\definecolor{bricks}{HTML}{E74C3C}
\definecolor{steelblue}{HTML}{4682b4}
\definecolor{goldenrod}{HTML}{daa520}
\definecolor{fista}{HTML}{075187}
\definecolor{fistagray}{HTML}{929598}
\colorlet{steve}{blue!80!gray!80!green} 
\definecolor{beige}{RGB}{245,245,220}

\newcommand{\defeq}{\vcentcolon=}

\DeclareMathOperator*{\argmax}{\operatorname{argmax}}

\DeclareMathOperator*{\blkdiag}{\operatorname{blkdiag}}

\newcommand{\thetitlecolor}{black}
\newcommand{\themaintitle}{Default Main Title}

\newcommand{\titlecolor}[1]{\renewcommand{\thetitlecolor}{#1}}
\newcommand{\maintitle}[1]{\renewcommand{\themaintitle}{#1}}

\newcommand{\ddto}{\textsc{ddto}}

\newcommand{\vrep}{\textsc{V-rep}}
\newcommand{\hrep}{\textsc{H-rep}}
\newcommand{\cgrep}{\textsc{CG-rep}}
\newcommand{\vrepheading}{\textbf{V-{\relsize{-2}REP}}}
\newcommand{\hrepheading}{\textbf{H-{\relsize{-2}REP}}}
\newcommand{\cgrepheading}{\textbf{CG-{\relsize{-2}REP}}}

\newcommand{\tf}{t_{\mathrm{f}}}

\newcommand{\mdry}{m_{\mathrm{dry}}}
\newcommand{\mwet}{m_{\mathrm{wet}}}

\newcommand{\wx}{w^{x}}

\newcommand{\wu}{w^{u}}
\newcommand{\Wu}{\mathcal{W}_{u}}

\newcommand{\N}{\mathcal{N}}

\newcommand{\psiu}{\psi_{u}}

\newcommand{\Ac}{A_{\mathrm{c}}}
\newcommand{\Bc}{B_{\mathrm{c}}}
\newcommand{\Gc}{G_{\mathrm{c}}}

\newcommand{\gc}{d_{\mathrm{c}}}

\newcommand{\Ad}{A}
\newcommand{\Bd}{B}

\newcommand{\gd}{d}

\newcommand{\CS}{\textsc{cs}}

\newcommand{\R}{\mathbb{R}}

\newcommand{\X}{\mathcal{X}}

\newcommand{\U}{\mathcal{U}}
\newcommand{\W}{\mathcal{W}}

\newcommand{\CZ}{\mathcal{Z}}

\newcommand{\argmin}{\operatorname*{argmin}}
\newcommand{\ming}{\operatorname*{min\vphantom{g}}}

\newtcolorbox{mybox}
{
  enhanced jigsaw,
  colframe=black,
  colback=white,
  boxrule=0.75pt
}

\newtcolorbox{mybox2}
{
  enhanced jigsaw,
  colframe=black,
  colback=white,
  boxrule=0.75pt,
  hbox
}

\newtheorem{theorem}{Theorem}

\newtheorem{assumption}{Assumption}
\newtheorem*{notation}{Notation}

\newcounter{problem}
\renewcommand{\theproblem}{P\arabic{problem}}

\makeatletter
\newcounter{problemeqsave}
\newcommand\problem@oldtheequation{}
\newcommand\problem@oldp@equation{}
\newenvironment{problem}{%
  \refstepcounter{problem}%
  \setcounter{problemeqsave}{\value{equation}}%
  \setcounter{equation}{0}%
  \let\problem@oldtheequation\theequation
  \let\problem@oldp@equation\p@equation
  \renewcommand{\theequation}{\theproblem\problem@letter}%
  \renewcommand{\p@equation}{}%
  \newcommand{\problem@letter}{%
    \ifnum\value{equation}=1\relax
    \else
      \@alph{\numexpr\value{equation}-1\relax}%
    \fi
  }%
}{%
  \setcounter{equation}{\value{problemeqsave}}%
  \let\theequation\problem@oldtheequation
  \let\p@equation\problem@oldp@equation
}
\makeatother

\makeatletter
\newcommand\blankfootnote[1]{%
  \begingroup
    \renewcommand\thefootnote{}%
    \long\def\@makefntext##1{\parindent0pt\noindent ##1}%
    \insert\footins{\vskip 0.5\baselineskip}%
    \@ifpackageloaded{hyperref}{\NoHyper\footnotetext{#1}\endNoHyper}{\footnotetext{#1}}%
  \endgroup
}
\makeatother

%% file: preamble/switch.tex
\ifjgcd
\else

    \allowdisplaybreaks
    
    \usepackage[pdfencoding=auto]{hyperref}
    \usepackage[bottom, symbol]{footmisc}
    \usepackage[square, numbers]{natbib}

    \BeforeBeginEnvironment{algorithm}{\begin{spacing}{1}}
    \AfterEndEnvironment{algorithm}{\end{spacing}\vspace{-1.5\parskip}}

\fi

%% file: preamble/settings.tex
\titlecolor{black}

\maintitle{Set-based Optimal, Robust, and Resilient Control\\[0.125em] {\large{with Applications to Autonomous Precision Landing}}}

\ifjgcd
\else
    \author{Abhinav G.\ Kamath\footnote{Ph.D.\ Candidate, William E.\ Boeing Department of Aeronautics \& Astronautics, University of Washington; \hphantom{Research} Research Intern (during the development of this work), Mitsubishi Electric Research Laboratories; \texttt{agkamath@uw.edu}} }
\fi

\author{Abraham P.\ Vinod\footnote{Principal Research Scientist, Mitsubishi Electric Research Laboratories; \texttt{abraham.p.vinod@ieee.org}} }

\author{Purnanand Elango\footnote{Research Scientist, Mitsubishi Electric Research Laboratories; \texttt{elango@merl.com}} }

\author{Stefano Di Cairano\footnote{Distinguished Research Scientist, Mitsubishi Electric Research Laboratories; \texttt{dicairano@merl.com}} }

\author{Avishai Weiss\footnote{Senior Principal Research Scientist, Mitsubishi Electric Research Laboratories; \texttt{weiss@merl.com}}}

\affil{}

%% file: preamble/setup.tex
\hypersetup{
    colorlinks=true,
    linkcolor=darkgray,
    citecolor=teal,
    filecolor=PineGreen,      
    urlcolor=BurntOrange,
    pdfpagemode=FullScreen,
}

\tcbuselibrary{skins}

\newtcolorbox{fullbox}
{
  enhanced jigsaw,
  colframe=black,
  colback=white,
  drop shadow=black!50!white,
  boxrule=0.75pt
}

\newtcolorbox{cropbox}
{
  enhanced jigsaw,
  colframe=black,
  colback=white,
  drop shadow=black!50!white,
  boxrule=0.75pt,
  hbox
}

\title{\Large\textcolor{\thetitlecolor}{\textbf{\themaintitle}}\vspace{1em}}

\date{}

\immediate\write18{cp `kpsewhich natdin.bst` .}

\newtheorem{remark}{Remark}

%% file: preamble/titlepage.tex

\maketitle

\topskip0pt
\vspace*{\fill}
\begin{abstract}
    \input{sections/abstract}
\end{abstract}
\vspace*{\fill}%
\thispagestyle{empty}

\ifjgcd
\else
    \newpage
    \begin{spacing}{1.25}
    {\footnotesize\tableofcontents}
    \end{spacing}
    \newpage
\fi

%% file: sections/abstract.tex
\noindent We present a real-time-capable set-based framework for closed-loop predictive control of autonomous systems using tools from computational geometry, dynamic programming, and convex optimization. The control architecture relies on the offline precomputation of the controllable tube, i.e, a time-indexed sequence of controllable sets\ifjgcd\else{, which are sets that contain all possible states that can reach a terminal set under state and control constraints}\fi{}. Sets are represented using constrained zonotopes (CZs), which are efficient encodings of convex polytopes that support fast set operations and enable tractable dynamic programming in high dimensions. Online, we obtain a globally optimal control profile\ifjgcd\else{ via a forward rollout, i.e.,}\fi{} by solving a series of one-step optimal control problems\ifjgcd\else{, each of which takes the current state to the next controllable set in the tube}\fi{}. Our key contributions are: (1) free-final-time optimality: we devise an optimal horizon computation algorithm to achieve global optimality; (2) robustness: we handle stochastic uncertainty in both the state and control, with probabilistic guarantees, by constructing bounded disturbance sets; (3) resilience: we develop (i) an optimization-free approach to computing the instantaneous reachable set, i.e., the reachable set from the current state, to enable, for example, large/maximal divert maneuvers, and (ii) an approach to achieving maximal decision-deferral, i.e., maintaining reachability/divert-feasibility to multiple targets for as long as possible.\ifjgcd\else{ The optimal and resilient control approaches we propose are exact (approximation-free) for optimal control problems with polytopic feasible sets, and conservative in the right direction for their robust variants.}\fi{} By means of an autonomous precision landing case study, we demonstrate globally optimal free-final-time guidance, robustness to navigation and actuation uncertainties, instantaneous divert envelope computation, and maximal decision-deferral.

%% file: sections/introduction.tex
\section{Introduction}

Modern autonomous systems require advanced real-time-capable control methods. The control problem of autonomous precision landing is particularly challenging owing to payload capacity coming at a premium, estimation and control errors due to uncertainty in navigation and actuation, and the lack of safety due to incomplete or outdated information about the landing zone \citep{starek2015spacecraft,blackmore2016autonomous,carson2019splice,malyuta2021advances}. At the same time, the overarching goal of autonomous precision landing missions is to maximize the payload capacity by minimizing propellant consumption (optimal control), while remaining robust to navigation and actuation uncertainty (robust control), and maintaining safety in the presence of unforeseen hazards (resilient control). These competing requirements naturally lead to the need for a unified control architecture.

In this work, motivated by the autonomous precision landing problem, we develop a computationally tractable and real-time-capable approach to optimal, robust, and resilient predictive control—for a class of systems—within a unified control architecture. Our approach is set-theoretic: we present general results for closed-loop optimal control, for control under modeled state and control uncertainty, and for control in the presence of unmodeled uncertainty. We develop a computational framework for closed-loop control via dynamic programming over a controllable tube, i.e., a time-indexed sequence of controllable sets \citep{vinod2025set}. For computational tractability, we leverage constrained zonotopes, a representation of compact convex polytopes that admits efficient and scalable set operations \citep{scott2016constrained}.

Robust open-loop control with both state and control uncertainty can be prohibitively conservative, motivating the need for a closed-loop control framework \citep{borrelli2017predictive}. Global optimality in a closed-loop setting is naturally characterized by dynamic programming \citep{bellman1954theory}. Standard approaches to dynamic programming, however, require discretization of the state space and suffer from the curse of dimensionality, making them intractable for high-dimensional systems \citep{bertsekas2012dynamic}. Recently, it was shown that for a class of nonlinear optimal control problems, dynamic programming can be made tractable by lossless convexification, a state-augmentation, and employing controllable tubes \citep{vinod2025set}, which we adopt in this work. The tractability of dynamic programming, consequently, relies on the tractability of controllable tube generation, for which we use constrained zonotopes \citep{scott2016constrained}.

The main contributions of this paper are: (i) a general set-theoretic result for closed-loop free-final-time control, (ii) robust control accounting for both control uncertainty and time-varying state uncertainty, (iii) a general set-theoretic result for computing the reachable set from the current state (i.e., the instantaneous reachable set), and (iv) set-based maximal decision-deferral, i.e., maintaining reachability of multiple targets for as long as possible, which enables large diverts; all with computationally tractable implementations, as demonstrated by means of a comprehensive autonomous precision landing case study.

Optimal predictive control for constrained systems is typically achieved by either: (i) model predictive control (MPC) and its variants \citep{garcia1989model, raimondo2009min, mayne2011tube, malyuta2019robust}, and (ii) dynamic programming (DP) and its variants \citep{bellman1954theory, bellman1966dynamic, bertsekas2012dynamic, powell2007approximate, powell2009you}. MPC involves the generation of an open-loop control sequence over a time-horizon, only applying the first control input, and repeating the process at each step \citep{garcia1989model, borrelli2017predictive, weiss2015model, accikmecse2011robust, malyuta2019robust}. Despite providing a feedback action to close the loop, the nominal prediction across the time-horizon itself is feedback-agnostic, i.e., it is open-loop, and robust variants tend to be conservative in terms of handling disturbances \citep{borrelli2017predictive}. DP, in contrast, yields globally optimal closed-loop policies, but standard techniques rely on state-space discretization and storage of value function evaluations for each state across the entire time-horizon, the number of which grows exponentially as a function of the state dimension, thus limiting its use to small-dimensional systems \citep{bertsekas2012dynamic}. Owing to the closed-loop nature of DP, robust control is less conservative than open-loop approaches.

Resilient control is an emerging field of control \citep{rieger2009resilient, chamon2020resilient} that deals with keeping the system under consideration safe in the face of unmodeled uncertainties, or, in other words, \emph{unknown unknowns}, such as disruptions and faults. Landing guidance algorithms with built-in divert capabilities serve as a good example for this \citep{acikmese2012g, acikmese2013flight, harris2014maximum, hayner2023halo, capolupo2024descent, srinivas2024lunar, lishkova2024divert}. Reachability-steering is an approach to resilient control that has been applied in the context of precision landing \citep{tomita2024optimal,tomita2025reachability}. Another framework for open-loop resilient control is that of \emph{deferred-decision trajectory optimization} (\ddto{}) \citep{elango2022deferring, hayner2023halo, elango2025deferred}, which ensures (constrained) reachability to a collection of target terminal states for as long as possible.

Set-based methods in control \citep{blanchini2015set,borrelli2017predictive} have recently emerged as a popular approach to deterministic optimal control \citep{eren2015constrained, srinivas2024lunar} and robust optimal control \citep{lishkova2024divert, vinod2024inscribing, vinod2025projection, vinod2025set}. Specifically, \citep{eren2015constrained, srinivas2024lunar, lishkova2024divert} propose open-loop approaches that involve the construction of a controllable set for the autonomous precision landing problem, and \citep{vinod2021stochastic, gleason2021lagrangian, vinod2021abort, vinod2025projection, pycvxset, vinod2025set} propose closed-loop approaches that involve the construction of a controllable \emph{tube}, i.e., a sequence of controllable sets, applied to spacecraft rendezvous trajectory optimization \citep{vinod2021stochastic, vinod2021abort, vinod2024inscribing, pycvxset, vinod2025projection}. Most set-theoretic approaches leverage polytopes, which are often limited to low dimensions due to computational challenges \citep{blanchini2015set}. However, the introduction of \emph{constrained zonotopes} (CZs) has enabled computationally tractable set operations in high dimensions \citep{scott2016constrained}, which are largely beneficial in the field of control \citep{scott2016constrained, vinod2024inscribing, vinod2025projection, vinod2025set}.

In the context of autonomous precision landing, while earlier work dealt with deterministic optimal control \citep{acikmese2007convex, accikmecse2013lossless}, there has been recent work on stochastic guidance and feedback control for landing as well \citep{ridderhof2018uncertainty, ridderhof2019minimum, ridderhof2021minimum}. More recently, \citep{srinivas2024lunar} proposed a deterministic open-loop control framework for lunar landing with divert capabilities, which was later extended to account for navigational uncertainty in \citep{lishkova2024divert}—these papers made use of a convex-optimization-based simplicial ``facet-pushing'' algorithm to construct the open-loop controllable set, first presented in \citep{eren2015constrained}. The case study we present herein serves as a closed-loop extension to the work on autonomous precision landing with the proposed novel contributions: closed-loop free-final-time optimal control, closed-loop control with robustness to both actuation (control) uncertainties and time-varying navigation (state) uncertainties, and instantaneous-reachable-set- and maximal decision-deferral-based resilient control.

\emph{Organization of the paper}: We present optimal control and set-based preliminaries in \S\ref{sec:preliminaries}. We then describe the set-based control architecture for closed-loop free-final-time optimal control in \S\ref{sec:free-final-time}, robust control in \S\ref{sec:robust}, and resilient control in \S\ref{sec:resilient}. Finally, we demonstrate the proposed framework by means of an autonomous precision landing case study in \S\ref{sec:landing}.

%% file: sections/preliminaries.tex
\section{Preliminaries} \label{sec:preliminaries}

In this section, we present the preliminaries that will aid in the development of the proposed computational framework in the paper. In \S\ref{subsec:deterministic}, we formulate the template deterministic optimal control problem and construct a corresponding discrete-time convex optimal control problem with a polytopic feasible set. In \S\ref{subsec:polytopic_cqc}, we present a computationally tractable approach to generating a polytopic inner approximation of a second-order (quadratic) cone, intersected with a halfspace, to convert second-order cone constraint sets—which show up in several control problems, including autonomous precision landing—to polytopes. Then, in \S\ref{subsec:CZ}, we define constrained zonotopes and describe the associated set operations. The use of constrained zonotopes enables the efficient generation of controllable tubes, which we describe in \S\ref{subsec:controllable-tube}. Finally, we present the one-step optimal control problem—that is executed online as part of the forward rollout—in \S\ref{subsec:one-step}.

\subsection{The Deterministic Optimal Control Problem} \label{subsec:deterministic}

We first consider a class of free-final-time optimal control problems in continuous-time, in the Lagrange form, with a convex cost functional, a class of nonlinear dynamical systems, and convex state and control constraints, as shown in Problem \ref{prob:ct_ocp_template_nonconvex}.
\ifjgcd
\else
    \vspace{\parskip}
\fi
\begin{problem}
\begin{mybox}
\begin{center}
    \underline{\textbf{Problem \ref{prob:ct_ocp_template_nonconvex}: Template Continuous-Time Optimal Control Problem (Nonconvex)}}
\end{center}
{\small
\begin{mini!}[2]
{\tf,\,u}{\makebox[3.25cm][l]{\textit{Cost Functional}:}\quad \int_{0}^{\tf} \beta\,h(u(t))\,dt}
{\label{prob:ct_ocp_template_nonconvex}}{}
\addConstraint{\makebox[3.25cm][l]{\textit{Dynamics}:}\quad}{\dot{x}(t) = \Ac\,x(t) + \Bc\,u(t) + \Gc\,h(u(t)) + \gc,}{\quad \forall t \in [0, \tf] \label{eq:ct_dynamics_nonlinear}}
\addConstraint{\makebox[3.25cm][l]{\textit{State Constraints:}}\quad}{x(t) \in \X_{\textsc{cvx}},}{\quad \forall t \in [0, \tf]}
\addConstraint{\makebox[3.25cm][l]{\textit{Control Constraints:}}\quad}{u(t) \in \U_{\textsc{cvx}},}{\quad \forall t \in [0, \tf]}
\addConstraint{\makebox[3.25cm][l]{\textit{Initial Condition:}}\quad}{x(0) = x_{\mathrm{i}}}
\addConstraint{\makebox[3.25cm][l]{\textit{Final Condition:}}\quad}{x(\tf) \in \X_{\mathrm{f}_{\textsc{cvx}}}}
\end{mini!}
}%
\end{mybox}
\end{problem}
\ifjgcd
\else
    \vspace{\parskip}
\fi
In Problem \ref{prob:ct_ocp_template_nonconvex}, $\tf \in \R_{++}$ is the final time (a free variable), $x(t) \in \R^{n_{x}}$ is the state, $u(t) \in \R^{n_{u}}$ is the control input, $\beta \in \R_{++}$ is a constant, and $h : \R^{n_{u}} \to \R$ is a convex function; $\X_{\textsc{cvx}} \subset \R^{n_{x}}$ and $\U_{\textsc{cvx}} \subset \R^{n_{u}}$ are the convex state and control constraint sets, respectively; $\X_{\mathrm{f}_{\textsc{cvx}}}$ is the convex final state set; and $x_{\mathrm{i}} \in \R^{n_{x}}$ is the fixed initial condition. In the dynamics, $A_{c} \in \R^{n_{x} \times n_{x}}$ is the system matrix, $B_{c} \in \R^{n_{x} \times n_{u}}$ is the control input matrix, $G_{c} \in \R^{n_{x}}$ is a vector that governs the influence of the convex function, $h$, on the state evolution, and $d_{c} \in \R^{n_{x}}$ is the affine term. All the constraint sets considered in this paper are assumed to be compact. Note that Problem \ref{prob:ct_ocp_template_nonconvex} is nonconvex due to the presence of convex functions, $h$, in the equality constraint given by Equation \eqref{eq:ct_dynamics_nonlinear}.

Problem \ref{prob:ct_ocp_template_nonconvex} can be relaxed to Problem \ref{prob:ct_ocp_template_convex}, which is convex.
\ifjgcd
\else
    \vspace{\parskip}
\fi
\begin{figure}[!htb]
\begin{problem}
\begin{mybox}
\begin{center}
    \underline{\textbf{Problem \ref{prob:ct_ocp_template_convex}: Template Continuous-Time Optimal Control Problem (Convexified)}}
\end{center}
{\small
\begin{mini!}[2]
{\tf,\,u,\,\sigma}{\makebox[3.25cm][l]{\textit{Cost Functional}:}\quad J \defeq \int_{0}^{\tf} \beta\,\sigma(t)\,dt\label{eq:cost-functional}}
{\label{prob:ct_ocp_template_convex}}{}
\addConstraint{\makebox[3.25cm][l]{\textit{Dynamics}:}\quad}{\dot{x}(t) = \Ac\,x(t) + \Bc\,u(t) + \Gc\,\sigma(t) + \gc,}{\quad \forall t \in [0, \tf]}
\addConstraint{\makebox[3.25cm][l]{\textit{Convex Relaxation:}}\quad}{h(u(t)) \le \sigma(t),}{\quad \forall t \in [0, \tf] \label{eq:convex-relaxation}}
\addConstraint{\makebox[3.25cm][l]{\textit{State Constraints:}}\quad}{x(t) \in \X_{\textsc{cvx}},}{\quad \forall t \in [0, \tf]}
\addConstraint{\makebox[3.25cm][l]{\textit{Control Constraints:}}\quad}{(u(t), \sigma(t)) \in \hat{\U}_{\textsc{cvx}},}{\quad \forall t \in [0, \tf]}
\addConstraint{\makebox[3.25cm][l]{\textit{Initial Condition:}}\quad}{x(0) = x_{\mathrm{i}}}
\addConstraint{\makebox[3.25cm][l]{\textit{Final Condition:}}\quad}{x(\tf) \in \X_{\mathrm{f}_{\textsc{cvx}}}}
\end{mini!}
}%
\end{mybox}
\end{problem}
\end{figure}
\ifjgcd
\else
    \vspace{\parskip}
\fi
In Problem \ref{prob:ct_ocp_template_convex}, $\sigma(t) \in \R$, which will be treated as an auxiliary control input, is the slack variable introduced to convexify the problem, and $\hat{\U}_{\textsc{cvx}} \subset \R^{n_{u} + 1}$ is the augmented control constraint set. Note that Problem \ref{prob:ct_ocp_template_convex} bears resemblance to the problem templates considered in the vast body of literature on lossless convexification in optimal control \citep{blackmore2012lossless, malyuta2022convex, vinod2025set}, which are shown to recover globally optimal solutions to the original nonconvex problems.

\paragraph{Cost-To-Go Augmentation} Recently, in the discrete-time setting, a general version of Problem \ref{prob:ct_ocp_template_nonconvex} was shown to be losslessly convexifiable under some conditions \citep{vinod2025set}. Specifically, \citep{vinod2025set} leverages dynamic programming for the synthesis of the online controller to achieve global optimality in a closed loop using the so-called \emph{controllable tube} (see \S\ref{subsec:controllable-tube}). A key insight in \citep{vinod2025set} is that the controllable tube must explicitly account for the cost-to-go, which can be subsequently minimized in the forward rollout. Next, we discuss augmentation of the dynamical system with the cost-to-go dynamics.

From the cost functional in Equation \eqref{eq:cost-functional}, $J$, the cost-to-go at any $t \in [0, \tf]$ is given by:
\begin{align}
    c(t) \defeq \int_{t}^{\tf} \beta\,\sigma(\tau)\,d\tau \label{eq:ct_cost_to_go}
\end{align}
Note that $J = c(0)$. Taking the time-derivative of the cost-to-go yield the auxiliary dynamical system:
\begin{align}
    \dot{c}(t) &= \frac{d}{dt}\int_{t}^{\tf} \beta\,\sigma(\tau)\,d\tau = \frac{d}{dt}\int_{\tf}^{t} -\beta\,\sigma(\tau)\,d\tau = -\beta\,\sigma(t) \label{eq:ct-cost-to-go-dynamics}
\end{align}
From Equation \eqref{eq:ct_cost_to_go}, the boundary conditions for this auxiliary dynamical system are given by:
\begin{subequations}
\begin{align}
    c(0) &= \int_{0}^{\tf} \beta\,\sigma(t)\,dt = J = \text{free}\label{eq:cost_to_go_initial_condition} \\
    c(\tf) &= 0
\end{align}
\end{subequations}
Augmenting the dynamical system with the \emph{cost-to-go dynamics} (Equation \eqref{eq:ct-cost-to-go-dynamics}) and writing the cost functional in terms of the cost-to-go state yields Problem \ref{prob:ct_ocp_template_augmented}, which is equivalent to Problem \ref{prob:ct_ocp_template_convex}.
\ifjgcd
\else
    \vspace{\parskip}
\fi
\begin{problem}
\begin{mybox}
\begin{center}
    \underline{\textbf{Problem \ref{prob:ct_ocp_template_augmented}: Template Continuous-Time Optimal Control Problem (Augmented)}}
\end{center}
{\small
\begin{mini!}[2]
{\tf,\,u,\,\sigma,\,c(0)}{\makebox[3.25cm][l]{\textit{Cost Functional}:}\quad c(0)}
{\label{prob:ct_ocp_template_augmented}}{}
\addConstraint{\makebox[3.25cm][l]{\textit{Dynamics}:}\quad}{\dot{x}(t) = \Ac\,x(t) + \Bc\,u(t) + \Gc\,\sigma(t) + \gc,}{\quad \forall t \in [0, \tf]}
\addConstraint{\makebox[3.25cm][l]{\textit{Cost-to-go Dynamics}:}\quad}{\dot{c}(t) = -\beta\,\sigma(t),}{\quad \forall t \in [0, \tf]}
\addConstraint{\makebox[3.25cm][l]{\textit{State Constraints:}}\quad}{(x(t), c(t)) \in \tilde{\X}_{\textsc{cvx}},}{\quad \forall t \in [0, \tf]}
\addConstraint{\makebox[3.25cm][l]{\textit{Control Constraints:}}\quad}{(u(t), \sigma(t)) \in \tilde{\U}_{\textsc{cvx}},}{\quad \forall t \in [0, \tf]}
\addConstraint{\makebox[3.25cm][l]{\textit{Initial Condition:}}\quad}{x(0) = x_{\mathrm{i}}}
\addConstraint{\makebox[3.25cm][l]{\textit{Final Condition:}}\quad}{x(\tf) \in \X_{\mathrm{f}_{\textsc{cvx}}},\; c(\tf) = 0}
\end{mini!}
}%
\end{mybox}
\end{problem}
\ifjgcd
\else
    \vspace{\parskip}
\fi
In Problem \ref{prob:ct_ocp_template_augmented}, $\tilde{\X}_{\textsc{cvx}} \subset \R^{n_{x}+1}$ is the augmented state constraint set and $\tilde{\U}_{\textsc{cvx}} \subset \R^{n_{u}+1}$ is the augmented control constraint set, subsuming the constraint set defined by the convex relaxation in Equation \eqref{eq:convex-relaxation}.

We adopt a zero-order hold (ZOH), i.e., a piecewise-constant control parameterization, and time-discretize Problem \ref{prob:ct_ocp_template_augmented}. Assuming the true optimal control is a continuous signal, the control parameterization may lead to suboptimality, but since the dynamics themselves are exactly discretized, the continuous-time trajectory will exactly pass through the discrete temporal nodes. The state and control constraint sets are only considered at the nodes. Given (i) the convexity of the control constraint sets, and (ii) the ZOH control parameterization, the control constraints are guaranteed to be satisfied at and in between temporal nodes. For the state, however, these constraints are only imposed at (finitely many) temporal nodes and inter-sample constraint satisfaction is not guaranteed in general. The resulting discrete-time optimal control problem is shown in Problem \ref{prob:dt_ocp_template_conic}.
\ifjgcd
\else
    \vspace{\parskip}
\fi
\begin{problem}
\begin{mybox}
\begin{center}
    \underline{\textbf{Problem \ref{prob:dt_ocp_template_conic}: Template Discrete-Time Optimal Control Problem (Conic)}}
\end{center}
{\footnotesize
\begin{mini!}[2]
{\substack{N,\,c_{1}, \\u_{1},\,\ldots,\,u_{N-1}, \\ \sigma_{1},\, \ldots,\,\sigma_{N-1}}}{\makebox[3cm][l]{\textit{Cost Function}:}\quad c_{1}}
{\label{prob:dt_ocp_template_conic}}{}
\addConstraint{\makebox[3cm][l]{\textit{Dynamics}:}\quad}{(x_{k+1}, c_{k+1}) = \Ad\cdot(x_{k}, c_{k}) + \Bd\cdot(u_{k}, \sigma_{k}) + \gd,}{\quad k = 1, \ldots, N-1}
\addConstraint{\makebox[3cm][l]{\textit{State Constraints:}}\quad}{(x_{k}, c_{k}) \in \tilde{\X}_{\textsc{cvx}},}{\quad k = 1, \ldots, N}
\addConstraint{\makebox[3cm][l]{\textit{Control Constraints:}}\quad}{(u_{k}, \sigma_{k}) \in \tilde{\U}_{\textsc{cvx}},}{\quad k = 1, \ldots, N-1}
\addConstraint{\makebox[3cm][l]{\textit{Initial Condition:}}\quad}{x_{1} = x_{\mathrm{i}}}
\addConstraint{\makebox[3cm][l]{\textit{Final Condition:}}\quad}{(x_{N}, c_{N}) \in \tilde{\X}_{\mathrm{f}_{\textsc{cvx}}}}
\end{mini!}
}%
\end{mybox}
\end{problem}
\ifjgcd
\else
    \vspace{\parskip}
\fi
In Problem \ref{prob:dt_ocp_template_conic}, $N$ is the number of temporal nodes (the horizon length), $\Ad$ and $\gd$ are the discrete-time counterparts of $\blkdiag\{\Ac, 0\} \in \R^{(n_{x}+1) \times (n_{x}+1)}$ and $(\gc, 0) \in \R^{n_{x}+1}$, respectively, and $\Bd$ is the discrete-time counterpart of $\begin{bmatrix}\Bc & \Gc \\ 0_{n_{u}} & -\beta \end{bmatrix} \in \R^{(n_{x}+1) \times (n_{u}+1)}$, and $\tilde{\X}_{\mathrm{f}_{\textsc{cvx}}} \defeq \X_{f} \times \{0\}$. We assume $A$ is given by a matrix exponential and is hence always invertible. Problem \ref{prob:dt_ocp_template_conic} is conic since $\tilde{\X}_{\textsc{cvx}}$ and $\tilde{\U}_{\textsc{cvx}}$ can be arbitrary conic constraint sets.

The polytopic approximation of Problem \ref{prob:dt_ocp_template_conic} is given by Problem \ref{prob:dt_ocp_template_polytopic}.
\ifjgcd
\else
    \vspace{\parskip}
\fi
\begin{problem}
\begin{mybox}
\begin{center}
    \underline{\textbf{Problem \ref{prob:dt_ocp_template_polytopic}: Template Discrete-Time Optimal Control Problem (Polytopic)}}
\end{center}
{\footnotesize
\begin{mini!}[2]
{\substack{N,\\u_{1},\,\ldots,\,u_{N-1}, \\ \sigma_{1},\, \ldots,\,\sigma_{N-1}}}{\makebox[3cm][l]{\textit{Cost Function}:}\quad c_{1}}
{\label{prob:dt_ocp_template_polytopic}}{}
\addConstraint{\makebox[3cm][l]{\textit{Dynamics}:}\quad}{(x_{k+1}, c_{k+1}) = \Ad\cdot(x_{k}, c_{k}) + \Bd\cdot(u_{k}, \sigma_{k}) + \gd,}{\quad k = 1, \ldots, N-1 \label{eq:dt-dynamics-augmented}}
\addConstraint{\makebox[3cm][l]{\textit{State Constraints:}}\quad}{(x_{k}, c_{k}) \in \tilde{\X},}{\quad k = 1, \ldots, N}
\addConstraint{\makebox[3cm][l]{\textit{Control Constraints:}}\quad}{(u_{k}, \sigma_{k}) \in \tilde{\U},}{\quad k = 1, \ldots, N-1}
\addConstraint{\makebox[3cm][l]{\textit{Initial Condition:}}\quad}{x_{1} = x_{\mathrm{i}},\enskip\: c_{1} = \text{free}}
\addConstraint{\makebox[3cm][l]{\textit{Final Condition:}}\quad}{(x_{N}, c_{N}) \in \tilde{\X}_{\mathrm{f}}}
\end{mini!}
}%
\end{mybox}
\end{problem}
\ifjgcd
\else
    \vspace{\parskip}
\fi
In Problem \ref{prob:dt_ocp_template_polytopic}, $\tilde{\X}$, $\tilde{\U}$, and $\tilde{\X}_{\mathrm{f}}$ are the polytopic approximations of $\tilde{\X}_{\textsc{cvx}}$, $\tilde{\U}_{\textsc{cvx}}$, and $\tilde{\X}_{\mathrm{f}_{\textsc{cvx}}}$, respectively. Note that there is no approximation involved if a constraint set is polytopic to begin with.

The optimal and resilient control approaches we describe in this paper are \emph{exact} for the template polytopic optimal control problem given by Problem \ref{prob:dt_ocp_template_polytopic}. However, in the realm of optimal control, most problems are better represented by the template of Problem \ref{prob:ct_ocp_template_nonconvex}, which may involve second-order cone constraints that need to be approximated to fit the template of Problem \ref{prob:dt_ocp_template_polytopic}, as shown. In \S\ref{subsec:polytopic_cqc}, we provide a numerically tractable method to obtain ``good'' polytopic approximations to these convexified conic programs (Problem \ref{prob:dt_ocp_template_conic}) for the special case wherein the conic constraint sets are quadratic cones, which are also known as standard/unit second-order cones (SOCs), ice-cream cones, or Lorentz cones \citep{lobo1998applications}.

\subsection{Polytopic Inner Approximation of a Compact Quadratic Cone} \label{subsec:polytopic_cqc}

\vspace{\parskip}
\begin{figure}[!htb]
    \centering
    \begin{mybox2}
    \includegraphics[width=0.4\linewidth]{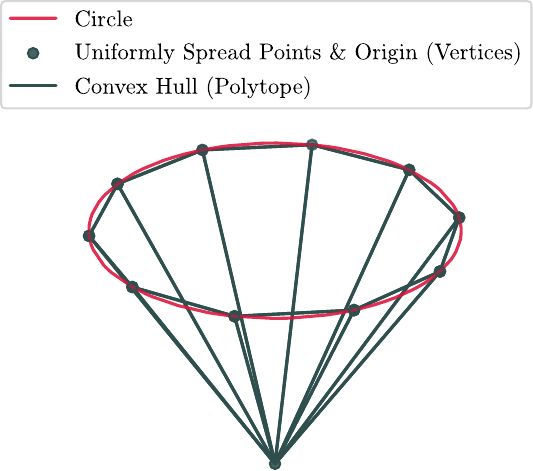}
    \end{mybox2}
    \caption{Polytopic inner approximation of a $3$-dimensional compact quadratic cone.}
    \label{fig:CQC-approximation}
\end{figure}
For controllable tube generation as described in \S\ref{subsec:controllable-tube}, the constraint sets are required to be polytopic, i.e., convex, closed, \emph{and} bounded. There exist methods in the literature to obtain polyhedral \emph{outer} approximations of an $n$-dimensional (unbounded, i.e., non-compact) quadratic cone \citep{ben2001polyhedral, vinel2014polyhedral}. However, in order to ensure that the generated controllable tube is conservative rather than infeasible, i.e., to guarantee feasibility of all points in the set, the polytopic constraint set must be an \emph{inner} approximation of the original constraint set.

Motivated by these reasons, we propose a computationally tractable approach to obtaining a polytopic \emph{inner} approximation of a closed and bounded, i.e., compact, $n$-dimensional quadratic cone, which is a special case of a quadratic cone, intersected with a halfspace.

We define a compact quadratic cone (CQC) as follows:
\begin{align}
    \mathbb{L}_{\textsc{cqc}}^{n} \defeq \left\{(z,\,t) \bigm| \norm{z}_{2} \le t,\ 0 \le t \le t_{\max} \right\}
\end{align}
where $z \in \R^{n-1}$ and $t \in \R$.

To obtain a polytopic approximation of $\mathbb{L}_{\textsc{cqc}}^{n}$, we take a slice of $\mathbb{L}_{\textsc{cqc}}^{n}$ along $t = t_{\max}$:
\begin{align}
    \mathbb{\widetilde{L}}_{\textsc{cqc}}^{n} \defeq \left\{(z,\,t) \bigm| \norm{z}_{2} \le t_{\max},\ t = t_{\max}\right\}
\end{align}
Here, $\mathcal{Z} \defeq \left\{z \bigm| \norm{z}_{2} \le t_{\max}\right\}$ defines an $(n-1)$-dimensional hyperball.

Now, we proceed to uniformly spread points on the boundary of the $(n-1)$-dimensional hyperball by (approximately) solving a difference-of-convex (DC) optimization problem via the convex-concave procedure\footnote{This is for $n \ge 3$; for $n \le 2$, the uniform spreading problem can be solved in closed-form.} \citep{pycvxset, gleason2021lagrangian, lipp2016variations}. The accuracy of the approximation can be improved by increasing the number of points. These $(n-1)$-dimensional points are then lifted to $n$ dimensions with $t = t_{\max}$.

The convex hull of the union of these points and the origin (in $n$ dimensions) is a polytope that is guaranteed to be an inner approximation of the original CQC. Figure \ref{fig:CQC-approximation} shows the polytopic approximation of a $3$-dimensional CQC with $10$ points spread on the boundary of the $2$-dimensional hyperball, i.e., a circle.

\subsection{Constrained Zonotopes} \label{subsec:CZ}

Polytopes are sets with a finite number of facets (flat sides). Convex polytopes that are closed and bounded, i.e., compact, can be represented either as the convex hull of a finite number of points/vertices (vertex representation: \vrep{}) or as the intersection of a finite number of halfspaces (halfspace representation: \hrep{}). In this work, we use the term \emph{polytopes} to refer to compact convex polytopes.

Zonotopes are sets that are given by the Minkowski sum of line segments, called generators, with unit-box-constrained weights, known as latent variables, which are typically higher-dimensional than the set itself. \emph{Constrained} zonotopes (CZs) are zonotopes with additional affine equality constraints on the weights (constrained generator representation: \cgrep{})—they are a recently introduced class of sets that can be used to describe arbitrary polytopes, with distinct computational advantages over standard representations (\vrep{}, \hrep{}) \citep{scott2016constrained}. In fact, a set is a constrained zonotope if and only if it is a polytope \citep[Theorem 1]{scott2016constrained}. While \vrep{} and \hrep{} polytopes suffer from the curse of dimensionality with exponential blow-up \citep{tiwary2008hardness}, \cgrep{} CZs do not (see Table \ref{tab:cz-v-polytope}). This key characteristic of CZs, in addition to the efficiency and accuracy of the corresponding set operations, has led to significant adoption in set-based methods, especially in the fields of estimation and control \citep{scott2016constrained, vinod2025projection}.

\begin{table}[H]
\centering
\small
\begin{tabular}{cccc}
\toprule
\textbf{Operation} & \textbf{\vrepheading{} Polytope} & \textbf{\hrepheading{} Polytope} & \textbf{\cgrepheading{} CZ} \\
\midrule
Intersection & ${\color{bricks}{\cross}}$ & ${\color{darts}{\checkmark}}$ & ${\color{darts}{\checkmark}}$ \\
\midrule
Minkowski Sum & ${\color{bricks}{\cross}}$ & ${\color{bricks}{\cross}}$ & ${\color{darts}{\checkmark}}$ \\
\midrule
Affine Map & ${\color{darts}{\checkmark}}$ & $\;{\color{bricks}{\cross}}^{\!\dagger}$ & ${\color{darts}{\checkmark}}$ \\
\bottomrule
\end{tabular}
\caption{Computational complexity for the required set operations in Algorithm \ref{alg:controllable_tube_recursion_deterministic}: ${\color{darts}{\checkmark}}$ indicates that the operation has polynomial complexity and ${\color{bricks}{\cross}}$ indicates that it has exponential complexity (${}^{\dagger}$exponential complexity in general, but polynomial complexity if the map is invertible).}
\label{tab:cz-v-polytope}
\end{table}
\ifjgcd
\else
    \vspace{-0.5\baselineskip}
\fi
A CZ in \cgrep{} is defined by the following set \citep{scott2016constrained, raghuraman2022set, vinod2025projection, pycvxset}:
\begin{align}
    \CZ \defeq \CZ(G, c, A, b) = \left\{x \mid \exists\,\xi,\ x = G\,\xi + c,\ \norm{\xi}_{\infty} \le 1,\ A\,\xi = b\right\} \label{eq:cz-definition}
\end{align}
where $x \in \R^{n}$, $\xi \in \R^{n_{g}}$, $G \in \R^{n \times n_{g}}$, $c \in \R^{n}$, $A \in \R^{n_{e} \times n_{g}}$, and $b \in \R^{n_{e}}$; $n$ is the dimension of the CZ, $n_{g}$ is the number of generators (columns of $G$), and $n_{e}$ is the number of equality constraints. Some useful set operations in terms of constrained zonotopes are presented in Table \ref{tab:set-operations}, where we use the following shorthand to denote a \cgrep{} CZ:
$\CZ_{i} \defeq \{G_{i},\,c_{i},\,A_{i},\,b_{i}\}$, where $i \in \mathbb{Z}_{++}$ is used to describe operations that involve more than one CZ. The subscript is dropped in the description of operations that only involve one CZ. A selector matrix is defined to be a matrix with one entry per row, which is unity.

\begin{table}[!htb]
\centering
\small
\begin{tabular}{m{3.925cm}m{5.875cm}m{6.125cm}}
\toprule
\textbf{Operation} & \textbf{Method} & \textbf{Description}\\
\midrule
\makecell[cl]{Affine Map \\ $R\,\CZ + r$} & \makecell[l]{$\left\{R\,G,\,R\,c + r,\,A,\,b\right\}$} & $R\,\CZ + r \defeq \{R\,x + r \mid x \in \CZ\}$, where $R \in \R^{n_{r} \times n}$ and $r \in \R^{n_{r}}$\\
\midrule
\makecell[cl]{Intersection \\ $\CZ_{1} \cap \CZ_{2}$} & $\!\left\{\begin{bmatrix}G_{1} & 0\end{bmatrix}\!,\,c_{1}, \begin{bmatrix}A_{1} & 0 \\ 0 & A_{2} \\ G_{1} & -G_{2}\end{bmatrix}\!, \begin{bmatrix}b_{1} \\ b_{2} \\ c_{2} - c_{1}\end{bmatrix}\right\}$ & $\CZ_{1} \cap \CZ_{2} \defeq \{x \in \CZ_{1} \mid x \in \CZ_{2}\}$ \\
\midrule
\makecell[cl]{Minkowski Sum \\ $\CZ_{1} \oplus \CZ_{2}$} & $\left\{\begin{bmatrix}G_{1} & G_{2}\end{bmatrix}\!,\,c_{1} + c_{2}, \begin{bmatrix}A_{1} & 0 \\ 0 & A_{2}\end{bmatrix}\!, \begin{bmatrix}b_{1} \\ b_{2}\end{bmatrix}\right\}$ & $\CZ_{1} \oplus \CZ_{2} \defeq \{x_{1} + x_{2} \mid x_{1} \in \CZ_{1},\, x_{2} \in \CZ_{2}\}$ \\
\midrule
\makecell[cl]{Pontryagin Difference \\ $\CZ_{1} \ominus \CZ_{2}$} & \makecell[cl]{\citep[Algorithm 3]{vinod2025projection} \\ (inner approximation)} & $\CZ_{1} \ominus \CZ_{2} \defeq \{x_{1} \mid \forall x_{2} \in \CZ_{2},\, x_{1} + x_{2} \in \CZ_{1}\}$ \\
\midrule
\makecell[cl]{Intersection with Affine Set \\ $\CZ \cap \{x \mid H\,x = h\}$} & $\left\{G,\,c, \begin{bmatrix}A \\ H\,G\end{bmatrix}\!, \begin{bmatrix}b \\ h - H\,c\end{bmatrix}\right\}$ & $x \in \R^{n}$, $H \in \R^{n_{h} \times n}$, and $h \in \R^{n_{h}}$ \\
\midrule
\makecell[cl]{Slice \\ $\CZ \cap \{x \mid E\,x = \bar{x}\}$} & $\begin{aligned}\left\{G,\,c, \begin{bmatrix}A \\ E\,G\end{bmatrix}\!, \begin{bmatrix}b \\ \bar{x} - E\,c\end{bmatrix}\right\}\end{aligned}$ & special case of intersection with affine set, where $\bar{x} \in \R^{n_{\bar{x}}}$, $n_{\bar{x}} \leq n$, and $E \in \R^{n_{\bar{x}} \times n}$ is a selector matrix \\
\midrule
\makecell[cl]{Projection \\ $E\,\CZ$} & $\left\{E\,G,\,E\,c,\,A,\,b\right\}$ & special case of affine map, where $E \in \R^{n_{\bar{x}} \times n}$ is a selector matrix, $n_{\bar{x}} \leq n$ \\
\midrule
\makecell[cl]{Distance} & $\underset{x\,\in\,\CZ}{\operatorname{minimize}}\ \norm{y - x}_{p}$ & $y \in \R^{n}$ is the point to be projected onto the set to compute the $p$-norm distance \\
\midrule
\makecell[cl]{Containment} & $\underset{x\,\in\,\CZ}{\operatorname{minimize}}\ \norm{y - x}_{p}$ & $y \in \CZ$ if the distance is zero; $y \notin \CZ$ otherwise \\
\midrule
\makecell[cl]{Emptiness} & $\underset{x\,\in\,\CZ}{\operatorname{minimize}}\ 0$ & the set is empty if the problem is infeasible; the set is nonempty otherwise \\
\midrule
\makecell[cl]{Support} & $\underset{x\,\in\,\CZ}{\operatorname{maximize}}\ \eta^{\top} x$ & $\eta \in \R^{n}$ is the direction vector along which to evaluate the support function \\
\midrule
\makecell[cl]{Extreme Point} & $\underset{x\,\in\,\CZ}{\argmax}\ \eta^{\top} x$ & point on the boundary of the set along the direction vector, $\eta$ \\
\bottomrule
\end{tabular}
\caption{{}Useful set operations in terms of constrained zonotopes, which can be found in \citep{scott2016constrained, raghuraman2022set, vinod2025projection, pycvxset}.}
\label{tab:set-operations}
\end{table}
\subsection{Controllable Tube Generation (Offline)}\label{subsec:controllable-tube}

We refer to the target final set as the \emph{terminal set}. We refer to the set of all initial conditions from which a given final set is reachable, in the presence of state and control constraints, for a given trajectory time, as a \emph{controllable set}. The union of all such controllable sets over time is referred to as the \emph{controllable tube}.

We start with deterministic optimal control problems. For problems that fit the template of Problem \ref{prob:dt_ocp_template_polytopic}, the exact controllable tube can be constructed by means of a simple backward set-recursion, given by Algorithm \ref{alg:controllable_tube_recursion_deterministic}, which is executed offline. Algorithm \ref{alg:controllable_tube_recursion_deterministic} is computationally tractable even in high dimensions when the sets in question are represented as constrained zonotopes (\cgrep{}). This can be extended to handle problems with uncertainties, which will be discussed in \S\ref{sec:robust}. Note that the generated controllable (backward reachable) sets are guaranteed to be feasible with respect to Problem \ref{prob:dt_ocp_template_conic}.
\vspace{\parskip}
\begin{algorithm}[!b]
\small
\caption{Set Recursion for Controllable Tube Generation}\label{alg:controllable_tube_recursion_deterministic}
    \vspace{1em}
    \begin{flushleft}
        \textbf{Inputs:} $N$, $\Ad$, $\Bd$, $\gd$, $\tilde{\mathcal{X}}$, $\tilde{\mathcal{X}}_{\mathrm{f}}$, $\tilde{\mathcal{U}}$
    \end{flushleft}
    \begin{algorithmic}[1]
    \vspace{1em}
    \State $\CS_{N}$ $\leftarrow$ $\tilde{\mathcal{X}}_{\mathrm{f}}$\vspace{0.5em}
    \For{$k = (N-1), (N-2), \ldots, 2, 1$}\vspace{0.25em}
    \State $\CS_{k} = \tilde{\mathcal{X}} \cap \Ad^{-1} \left(\CS_{k+1} \oplus \left(-\Bd\,\tilde{\mathcal{U}} - \gd\right) \right)$ \Comment{$\oplus$: Minkowski sum}\vspace{0.5em}
    \EndFor \vspace{1em}
    \end{algorithmic}
    \begin{flushleft}
        \textbf{Return:} $\CS_{1, \ldots, N}$\vspace{1em}
    \end{flushleft}
\end{algorithm}
Here, $N$ is the number of temporal nodes (the horizon length) and $\CS_{k}$ is the $k^\text{th}$ controllable set, the union over $k = 1, \ldots, N$ of which is the controllable tube. If the cost-to-go is strictly monotonically decreasing (and since the cost-to-go is bounded by construction), this set recursion is guaranteed to terminate for a finite $N$, i.e., the set recursion will yield empty sets after $N$ iterations, in which case the first nonempty controllable set, $\CS_{1}$, corresponds to the maximum-feasible trajectory time. In practice, since $N$ is not known a priori, this backward recursion is performed starting with index $1$ until it terminates at an index, $N$, after which the indices are reversed.

\subsection{One-Step Optimal Control (Online)} \label{subsec:one-step}

As shown in \citep{vinod2025set}, closed-loop optimal control synthesis may be done via a \emph{forward rollout}, i.e., by solving a sequence of one-step optimal control problems. The forward rollout involves minimization of the current cost-to-go, and one-step set-containment constraints: the next-step controllable set for the propagated state, and the current-step control constraint set for the control input. The optimal control problem—given by Problem \ref{prob:one_step_ocp_polytopic}—itself is a linear program (LP), and can be solved using efficient LP solvers \citep{meindl2012analysis, gearhart2013comparison, huangfu2018parallelizing, applegate2025pdlp}. 
\ifjgcd
\else
    \vspace{\parskip}
\fi
\begin{problem}
\begin{mybox}
\begin{center}
    \underline{\textbf{Problem \ref{prob:one_step_ocp_polytopic}: One-Step Optimal Control Problem (Polytopic)}}
\end{center}
{\small
\begin{mini!}[2]
{u_{k}, \sigma_{k}, c_{k}}{\makebox[5.5cm][l]{\textit{Cost Function}:}\quad c_{k}}
{\label{prob:one_step_ocp_polytopic}}{}
\addConstraint{\makebox[5.5cm][l]{\textit{Dynamics}:}\quad}{(x_{k+1}, c_{k+1}) = \Ad\cdot(x_{k}, c_{k}) + \Bd\cdot(u_{k}, \sigma_{k}) + \gd}{}
\addConstraint{\makebox[5.5cm][l]{\textit{State Constraints (Controllable Set):}}\quad}{(x_{k+1}, c_{k+1}) \in \CS_{k+1} \label{eq:containment_one_step_polytopic}}{}
\addConstraint{\makebox[5.5cm][l]{\textit{Control Constraints (Polytopic):}}\quad}{(u_{k}, \sigma_{k}) \in \tilde{\U}}{}
\end{mini!}
}%
\end{mybox}
\end{problem}
\ifjgcd
\else
    \vspace{\parskip}
\fi
\vspace{\parskip}
\begin{remark}
    The control constraint set, $\tilde{\U}$, in the one-step optimal control problem, Problem \ref{prob:one_step_ocp_polytopic}, need not be the polytopic approximation. Instead, we can use the original conic control input set—since it is larger than the polytopic approximation by construction—and solve Problem \ref{prob:one_step_ocp}. This has the two-fold benefit of: (i) rendering the convex relaxation tight with respect to the original conic constraint, and (ii) producing solutions with a lower cost. Consequently, instead of solving a sequence of $1$-step linear programs, we would solve a sequence of $1$-step conic programs, for which there exist efficient solvers \citep{domahidi2013ecos, yu2022extrapolated, goulart2024clarabel, chari2025qoco}. \label{remark:conic_one_step_ocp}
\end{remark}
\ifjgcd
\else
    \vspace{\parskip}
\fi
\begin{problem}
\begin{mybox}
\begin{center}
    \underline{\textbf{Problem \ref{prob:one_step_ocp}: One-Step Optimal Control Problem}}
\end{center}
{\small
\begin{mini!}[2]
{u_{k}, \sigma_{k}, c_{k}}{\makebox[5.5cm][l]{\textit{Cost Function}:}\quad c_{k}}
{\label{prob:one_step_ocp}}{}
\addConstraint{\makebox[5.5cm][l]{\textit{Dynamics}:}\quad}{(x_{k+1}, c_{k+1}) = \Ad\cdot(x_{k}, c_{k}) + \Bd\cdot(u_{k}, \sigma_{k}) + \gd}{}
\addConstraint{\makebox[5.5cm][l]{\textit{State Constraints (Controllable Set):}}\quad}{(x_{k+1}, c_{k+1}) \in \CS_{k+1} \label{eq:containment_one_step_conic}}{}
\addConstraint{\makebox[5.5cm][l]{\textit{Control Constraints (Conic):}}\quad}{(u_{k}, \sigma_{k}) \in \tilde{\U}_{\textsc{cvx}}}{}
\end{mini!}
}%
\end{mybox}
\end{problem}
\ifjgcd
\else
    \vspace{\parskip}
\fi
\vspace{\parskip}
\begin{remark}
    The recursive feasibility of Problems \ref{prob:one_step_ocp_polytopic} and \ref{prob:one_step_ocp} is guaranteed by Equations \eqref{eq:containment_one_step_polytopic} and \eqref{eq:containment_one_step_conic}, respectively.
\end{remark}

%% file: sections/free_final_time.tex
\section{Free-Final-Time Optimal Control} \label{sec:free-final-time}

\vspace{0.5\baselineskip}
\begin{figure}[htp!]
    \centering
    \begin{mybox2}
    \includegraphics[width=0.65\linewidth]{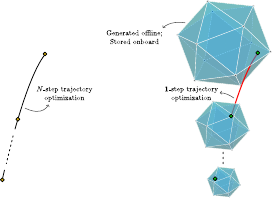}
    \end{mybox2}
    \caption{The open-loop approach (left) vs.\ the closed-loop approach using a controllable tube, i.e., a collection of controllable sets (right): the closed-loop approach (Algorithm \ref{alg:forward_rollout}) recovers a globally optimal solution to the open-loop problem (Problem \ref{prob:dt_ocp_template_polytopic}).}
    \label{fig:ol-vs-cl-fig}
\end{figure}

In this section, we present the first main result, a computationally tractable set-based approach to free-final-time optimal control. We present the computation of the optimal horizon length in \S\ref{subsec:optimal-horizon}, and the corresponding free-final-time forward rollout algorithm in \S\ref{subsec:forward-rollout}.

\subsection{Optimal Horizon Computation} \label{subsec:optimal-horizon}

With the existing controllable tube architecture, we can compute the free-final-time cost-optimal closed-loop trajectory, for a given initial condition, directly. The proposed approach requires simple and cheap set operations and the solution to small-dimensional linear programs. A notable contribution here is that this method can apply to problems that fit the template of Problem \ref{prob:dt_ocp_template_polytopic} in general, as long as the backward set-recursion terminates for a finite $N$; i.e., we make no assumptions on the relationship between the cost metric and the trajectory time. However, additional information about this relationship, if known, can be leveraged to accelerate the search. For example, if the relationship is known to be unimodal \citep{acikmese2005powered, acikmese2008enhancements}, a golden section search can be employed \citep{bertsekas1997nonlinear, kochenderfer2019algorithms}.

During online execution, the initial condition of the system may be contained in several of the controllable sets. In other words, there may exist feasible solutions to the terminal set starting from many different controllable sets—if the system can reach the terminal set in $M \le N$ steps (with $M-1$ control actions) for example, it could potentially reach the terminal set in $M+1$ steps (with $M$ control actions) and/or $M-1$ steps (with $M-2$ control actions), and so on, as well.

A natural approach then is to check for containment of the initial condition point in the controllable sets starting from the terminal set, going backward in time, and to then stop when the first containment is achieved. This would correspond to the minimum-time solution \citep{blanchini2015set}. However, rather than choosing the minimum-time controllable set (or any other time-slice based on such heuristics), we can instead choose the time-slice that corresponds to the least possible cost-to-go, which is easy to evaluate in real-time.

Now, we present a general set-theoretic result for free-final-time optimal control that we then leverage within our computational closed-loop control framework.

\vspace{0.5\baselineskip}
\begin{theorem}
     Let $x = x_{\mathrm{i}} \in \R^{n_{x}}$ be the initial condition (current state). Further, let:
     \begin{align}
        \mathcal{K} &\defeq \left\{k \in \{1, \ldots, N\} \mid x_{\mathrm{i}} \in \left\{x \mid (x,\,c) \in \CS_{k}\right\}\right\} \label{eq:containment-indices} \\
         \mathcal{C}_{k}(x_{\mathrm{i}}) &\defeq \{c \mid (x_{\mathrm{i}},\,c) \in \CS_{k}\}, \quad k \in \mathcal{K} \\
         c_{k} &\defeq \!\!\!\ming\limits_{\,c\,\in\,\mathcal{C}_{k}(x_{\mathrm{i}})} c
     \end{align}     
    where $\mathcal{K}$ is the set of indices of controllable sets $\CS_{k}$, $k = 1, \ldots, N$, that contain $x_{\mathrm{i}}$, and, for any $k \in \mathcal{K}$, $\mathcal{C}_{k}(x_{\mathrm{i}})$ is the slice of $\CS_{k}$ at $x_{\mathrm{i}}$ projected onto the cost-to-go coordinate, and $c_{k}$ is the minimum of $\mathcal{C}_{k}(x_{\mathrm{i}})$.  Assume $\mathcal{K}$ is nonempty and, for all $k \in \mathcal{K}$, $\CS_{k}$ is compact. Then the globally optimal horizon length is $N - k^{\star}$, where:
    \begin{align}
        k^{\star} = \argmin\limits_{k\,\in\,\mathcal{K}}\,c_{k}
    \end{align}
    and the corresponding globally optimal cost is:
    \begin{align}
        c^{\star} \defeq c_{k^{\star}}
    \end{align}
\end{theorem}
\begin{proof}
    Consider any $k \in \mathcal{K}$. Let the optimal trajectory cost to go from $x_{\mathrm{i}}$ to $\X_{\mathrm{f}}$ in $N - k$ steps be $J_{k}$. By construction, $(x,\,c) \in \CS_{k}$ if and only if there exists a sequence of $N - k$ feasible one-step control actions from $x$ to $\X_{\mathrm{f}}$, with cost less than or equal to $c$, i.e.:
    \begin{align}
        J_{k} &\le c, \enskip \forall\,c \in \mathcal{C}_{k}(x_{\mathrm{i}})
        \iff J_{k} \le \!\!\!\ming\limits_{\,c\,\in\,\mathcal{C}_{k}(x_{\mathrm{i}})} c 
        \iff J_{k} \le c_{k} \label{eq:cost_le_cost2go}
    \end{align}
    Next, since the last dimension in $\CS_{k}$ is the cost-to-go state (by construction), we have that $(x_{\mathrm{i}},\,J_{k}) \in \CS_{k}$, which implies $J_{k} \in \mathcal{C}_{k}(x_{\mathrm{i}})$. Therefore, we have:
    \begin{align}
        \!\!\!\ming\limits_{\,c\,\in\,\mathcal{C}_{k}(x_{\mathrm{i}})} c &\le J_{k}
        \iff c_{k} \le J_{k} \label{eq:cost2go_le_cost}
    \end{align}
    From Equations \eqref{eq:cost_le_cost2go} and \eqref{eq:cost2go_le_cost}, we have:
    \begin{align}
        J_{k} = c_{k} \label{eq:fixed-final-time-cost}
    \end{align}
    Taking the minimum of Equation \eqref{eq:fixed-final-time-cost} over $k \in \mathcal{K}$, we get:
    \begin{align}
        \ming\limits_{k\,\in\,\mathcal{K}}\,J_{k} = \ming\limits_{k\,\in\,\mathcal{K}}\,c_{k}
    \end{align}
    which is attained at $k^{\star} = \argmin_{k\,\in\,\mathcal{K}}\,c_{k}$. Hence, the globally optimal horizon length is $N - k^{\star}$ and the corresponding globally optimal cost is $c^{\star} = c_{k^{\star}}$.
\end{proof}

\vspace{0.5\baselineskip}
\begin{remark}
    If $\CS_{k}$, $k \in \mathcal{K}$, are compact convex sets, then each $\,\mathcal{C}_{k}(x_{\mathrm{i}})$ is a closed interval in $\R$, and $\,\min_{c\,\in\,\mathcal{C}_{k}(x_{\mathrm{i}})}\,c$ is its left endpoint.
\end{remark}

To compute the minimum cost-to-go, we first slice the $(n_{x}+1)$-dimensional controllable tube along the $n_{x}$-dimensional initial condition point, project it onto the cost-to-go coordinate (to obtain an interval, i.e., a set in one dimension), and then compute the left extreme point of the remaining one-dimensional cost-to-go set (interval). This value corresponds to the minimum possible cost-to-go from that particular controllable set. Note that this set, although $1$-dimensional, is a constrained zonotope, so computing the left extreme point requires solving a linear program.

We perform this computation for each of the time-slices that contain the initial condition, and pick the time-slice with the least cost-to-go among them. Then, the forward rollout algorithm, as described in \S\ref{subsec:forward-rollout}, starting from this set, will produce the globally optimal trajectory.
\begin{algorithm}[H]
\small
\caption{Optimal Horizon Computation}\label{alg:horizon}
    \vspace{1em}
    \begin{flushleft}
        \textbf{Inputs:} (1) Initial condition: $x_{\mathrm{i}}$\\
        \hphantom{\textbf{Inputs:}}\!\; (2) Controllable tube: $\CS_{1,\ldots,N}$
    \end{flushleft}
    \begin{algorithmic}[1]
    \vspace{1em}
    \State $\mathcal{K}$ $\gets$ \texttt{containment\_check}$(x_{\mathrm{i}}, \CS_{1,\ldots,N})$\Comment{Equation \eqref{eq:containment-indices}}\vspace{0.5em}
    \For{$k \in \mathcal{K}$}\vspace{0.25em}
    \State $\mathcal{C}_{k}$ $\gets$ slice $\CS_{k}$ along $x_{\mathrm{i}}$ and project it onto the cost-to-go coordinate\vspace{0.25em}
    \State $c_{k} \gets \ming\limits_{\,c\,\in\,\mathcal{C}_{k}} c$\vspace{0.25em}
    \EndFor\vspace{0.5em}
    \State $k^{\star} \gets \argmin\limits_{k\,\in\,\mathcal{K}}\,c_{k}$ \vspace{1em}
    \end{algorithmic}
    \begin{flushleft}
        \textbf{Return:} $k^{\star}$\Comment{optimal start index}\vspace{1em}
    \end{flushleft}
\end{algorithm}
Algorithm \ref{alg:horizon} gives us the optimal starting index, $k^{\star}$, i.e., the controllable set index from which to initiate the sequence of $N-k^{\star}$ control actions, which is guaranteed to give us the globally optimal trajectory. 

\subsection{Forward Rollout} \label{subsec:forward-rollout}

The \emph{forward rollout} of one-step optimal control problems, from the current controllable set (with index $k$) to the next (with index $k + 1$), is described in Algorithm \ref{alg:forward_rollout}. Note that this is distinct from an \emph{open-loop} approach, wherein the sequence of optimal control actions is determined in one shot—unlike that, the sequence of optimal control actions is determined one step at a time here, allowing for the incorporation of feedback at every step, making it a \emph{closed-loop} approach; see Figure \ref{fig:ol-vs-cl-fig} for a depiction of the same.
\begin{algorithm}[H]
\small
\caption{Forward Rollout}\label{alg:forward_rollout}
    \vspace{1em}
    \begin{flushleft}
        \textbf{Inputs:} (1) Initial condition: $x_{\mathrm{i}}$\\
        \hphantom{\textbf{Inputs:}}\!\; (2) Controllable tube: $\CS_{1,\ldots,N}$
    \end{flushleft}
    \begin{algorithmic}[1]
    \vspace{1em}
    \State $k^{\star}$ $\gets$ \texttt{optimal\_horizon}$(x_{\mathrm{i}}, \CS_{1,\ldots,N})$\Comment{Algorithm \ref{alg:horizon}}\vspace{0.5em}
    \State $x_{k^{\star}} \gets x_{\mathrm{i}}$\vspace{0.5em}
    \For{$k = k^{\star}\!, \dots, N-1$}\vspace{0.25em}
    \State $x_{k+1}, u_{k}$ $\gets$ \texttt{one\_step\_optimal\_control}$(x_{k}, \CS_{k+1})$ \Comment{solve Problem \ref{prob:one_step_ocp}}\vspace{0.25em}
    \EndFor\vspace{1em}
    \end{algorithmic}
    \begin{flushleft}
        \textbf{Return:} $u_{k^{\star}\!,\ldots,N-1}$\Comment{globally optimal control sequence}\vspace{1em}
    \end{flushleft}
\end{algorithm}

%% file: sections/robust.tex
\section{Robust Control}\label{sec:robust}

We describe the robust optimal control problem and distinguish between robust open-loop optimal control and robust closed-loop optimal control in \S\ref{subsec:robust-optimal-control}. We assume that the state measurements and the control input are subject to stochastic uncertainty, and further, that the uncertainties in the state are time-varying; we describe the modeling of these uncertainties in \S\ref{subsec:uncertainty-modeling}. The robust control approach we propose, however, requires the characterization of bounded disturbance sets, such that the closed-loop system is robust to any state and control disturbances that lie within these sets.

Rather than arbitrarily choosing bounded sets for this purpose, we are interested in the systematic construction of bounded sets that allow us to make the claim that the trajectory, even in the presence of random disturbance vectors sampled from known (possibly unbounded) probability distributions, is guaranteed to stay within the controllable tube with a user-specified probability. We describe this procedure in generality in \S\ref{subsec:stochastic}. In the case that the uncertainty is bounded to begin with, the corresponding bounded disturbance sets can be directly used within our framework.

In the presence of disturbances that are coupled in time, the construction of the bounded disturbance sets requires additional considerations. Once the disturbance sets are constructed, the state and control constraint sets, and in turn, the controllable tube, are robustified, i.e., tightened, such that they are robust to the worst-case disturbances from the bounded sets; we describe the construction of the bounded disturbance sets and the robustification of the control and state constraints in \S\ref{subsec:robust-controllable-tube}. Note that these steps are performed offline. Within the confines of this shrunken feasible space, the controller synthesized online is cost-optimal, making this approach one of \emph{robust optimal control}, in that we prioritize robustness to disturbances first, and then optimize performance among the admissible controls.

\input{sections/subsections/open_loop_v_closed_loop}

\subsection{State and Control Uncertainty Modeling} \label{subsec:uncertainty-modeling}

Let $y \defeq (x, c) \in \R^{n_{x} + 1}$ be the augmented state and $s \defeq (u, \sigma) \in \R^{n_{u} + 1}$ be the augmented controls.

Consider the following deterministic discrete-time dynamics (Equation \eqref{eq:dt-dynamics-augmented}) in terms of $y$ and $s$:
\begin{align}
    y_{k+1} = \Ad\,y_{k} + \Bd\,s_{k} + \gd, \quad k = 1, \dots, N-1 \label{eq:deterministic_dt_dynamics}
\end{align}
We first consider the case where some (or all) of the control inputs are subject to an additive disturbance that is Gaussian and independent and identically distributed (i.i.d.):
\begin{align}
    \tilde{s}_{k} \defeq s_{k} + E_{w}^{u}\,\wu_{k}, \quad \wu_{k} \sim \N(0, \Sigma^{u}), \quad k = 1, \dots, N-1 \label{eq:uncertain_controls}
\end{align}
where $\Sigma^{u} \in \mathbb{S}^{\tilde{n}_{u}}$ is the control disturbance covariance and $\tilde{n}_{u} \le n_{u} + 1$ is the dimension of the uncertain controls. We define an \emph{embedding matrix} to be a matrix that maps/embeds the vector it left-multiplies to/in a (potentially) higher dimension—it only has one entry per column, and that entry is unity (and thus an embedding matrix is full column rank); an embedding matrix is akin to the transpose of a selector matrix. Here, $E_{w}^{u} \in \R^{(n_{u}+1) \times \tilde{n}_{u}}$ is the uncertain-control embedding matrix, i.e., it lifts $\wu_{k}$ from $\tilde{n}_{u}$ dimensions to $n_{u} + 1$ dimensions.
For $k = 1,\dots, N-1$, we have:
\begin{subequations}
\begin{align}
    y_{k+1} &= \Ad\,y_{k} + \Bd\,\tilde{s}_{k} + \gd \\
            &= \Ad\,y_{k} + \Bd\,(s_{k} + E_{w}^{u}\,\wu_{k}) + \gd \\
            &= \Ad\,y_{k} + \Bd\,s_{k} + \Bd\,E_{w}^{u}\,\wu_{k} + \gd \label{eq:control_uncertain_system}
\end{align}
\end{subequations}
Now, additionally, consider the case where some (or all) of the state measurements, $\tilde{y}_{k}, k = 1,\ldots,N-1$, are uncertain with independent additive Gaussian disturbances with time-varying covariances:
\begin{align}
    \tilde{y}_{k} &\defeq y_{k} + E_{w}^{x}\,\wx_{k}, \quad \wx_{k} \sim \N(0, \Sigma^{x}_{k}), \quad k = 1, \dots, N \label{eq:uncertain_state}
\end{align}
where $\Sigma^{x}_{k} \in \mathbb{S}^{\tilde{n}_{x}}$, $k = 1, \ldots, N$, are the (time-varying) state covariances, where $\tilde{n}_{x} \le n_{x} + 1$ is the dimension of the uncertain component of the state and $E_{w}^{x} \in \R^{(n_{x} + 1) \times \tilde{n}_{x}}$ is the uncertain-state embedding matrix. Therefore, for $k = 1, \dots, N-1$, we have:
\begin{subequations}
\begin{align}
    \tilde{y}_{k+1} &= y_{k+1} + E_{w}^{x}\,\wx_{k+1} \\
            &= \Ad\,y_{k} + \Bd\,s_{k} + \Bd\,E_{w}^{u}\,\wu_{k} + \gd + E_{w}^{x}\,\wx_{k+1} \\
            &= \Ad\left(\tilde{y}_{k} - E_{w}^{x}\,\wx_{k}\right) + \Bd\,s_{k} + \Bd\,E_{w}^{u}\,\wu_{k} + \gd + E_{w}^{x}\,\wx_{k+1} \\
            &= \Ad\,\tilde{y}_{k} + \Bd\,s_{k} + \gd + (\Bd\,E_{w}^{u}\,\wu_{k} + E_{w}^{x}\,\wx_{k+1} - \Ad\,E_{w}^{x}\,\wx_{k})
\end{align}
\end{subequations}
Finally, we have:
\begin{align}
    \tilde{y}_{k+1} = \Ad\,\tilde{y}_{k} + \Bd\,s_{k} + \gd + w_{k} \label{eq:uncertain_dt_dynamics}
\end{align}

where, for $k = 1, \ldots, N-1$, we have:
\begin{align}
    w_{k} &\hphantom{:}= \Bd\,E_{w}^{u}\,\wu_{k} + E_{w}^{x}\,\wx_{k+1} - \Ad\,E_{w}^{x}\,\wx_{k} \hphantom{:}= \begin{bmatrix}
        \Bd\,E_{w}^{u} & E_{w}^{x} & -\Ad\,E_{w}^{x}
    \end{bmatrix} \begin{bmatrix}
        \wu_{k} \\ \wx_{k+1} \\ \wx_{k}
    \end{bmatrix} \defeq M\,\hat{w}_{k}
\end{align}
where $\hat{w}_{k} \defeq (\wu_{k}, \wx_{k+1}, \wx_{k}) \in \R^{\tilde{n}_{u} +\,2\,\tilde{n}_{x}}$ is the concatenation of the disturbance vectors influencing the dynamics, and $M \defeq \begin{bmatrix}
        \Bd\,E_{w}^{u} & E_{w}^{x} & -\Ad\,E_{w}^{x}
    \end{bmatrix} \in  \R^{(n_{x} + 1) \times (\tilde{n}_{u} +\,2\,\tilde{n}_{x})}$ is a linear map. Further, note that $w_{N} = E_{w}^{x}\,w_{N}^{x}$. Equation \eqref{eq:uncertain_dt_dynamics} represents the uncertain dynamical system.

\subsection{Representing Stochastic Uncertainty as a Bounded Disturbance Set} \label{subsec:stochastic}

For Gaussian disturbances, which have unbounded distributions, ellipsoids are the best option to represent bounded disturbance sets, since they are the natural sets of confidence regions due to the shape of the Gaussian kernel. Each bounded set is the $p$-level set of a Gaussian distribution, which, in turn, guarantees that the true state at a given step, in the presence of a disturbance sampled from this Gaussian distribution, will lie within the corresponding robust controllable set with probability at least $p$ \citep{gleason2021lagrangian}.

The logical interpretation of this is as follows:
\vspace{-0.25em}
\begin{enumerate}[label=(\roman*), itemsep=-0.25em]
    \item The probability that a random disturbance vector sampled from the Gaussian distribution lies within the corresponding bounded disturbance set is $p$ (by construction of the bounded disturbance set). \label{item:logic-1}
    \item The true state at a given step is guaranteed to lie within the corresponding controllable set for any disturbance from within the bounded disturbance set (by construction/robustification of the controllable tube). \label{item:logic-2}
    \item From \ref{item:logic-1} and \ref{item:logic-2}, the probability that the true state at a given step will lie inside the corresponding controllable set for a random disturbance sampled from the Gaussian distribution is at least $p$.
\end{enumerate}

We assume the instantaneous disturbance, $w \in \R^{n}$, is Gaussian, with mean $\mu \in \R^{n}$ and covariance $\Sigma \in \mathbb{S}^{n}$, i.e., $w \sim \N(\mu, \Sigma)$.

Consider the ellipsoid with shape matrix $\Sigma$, centered around $\mu \in \R^{n}$ and parameterized by size parameter $R^{2} \in [0, \infty)$ \citep{gleason2021lagrangian}:
\begin{align}
    \mathcal{E}(R^{2}; \mu, \Sigma) \defeq \{q \in \R^{n} \mid (q - \mu)^{\top} \Sigma^{-1} (q - \mu) \le R^{2}\} \label{eq:ellipsoid}
\end{align}
Now, let $\eta \in \R^{n}$ be a standard normal random vector, i.e., $\eta \sim \N(0, I_{n})$. Then, $w$ can be written in terms of $\eta$ as follows \citep{billingsley1995}:
\begin{align}
    w &= \Sigma^{\frac{1}{2}}\,\eta + \mu \label{eq:wk-in-terms-of-eta}
\end{align}
Then,
\begin{subequations}
\begin{align}
    w \in \mathcal{E}(R^{2}; \mu, \Sigma) &\implies (w - \mu)^{\top} \Sigma^{-1} (w - \mu) \le R^{2} \label{eq:wk-ellipsoid} \\
    &\implies \left(\Sigma^{\frac{1}{2}}\,\eta + \mu - \mu\right)^{\top} \Sigma^{-1} \left(\Sigma^{\frac{1}{2}}\,\eta + \mu - \mu\right) \le R{^2} \\
    &\implies \left(\eta^{\top}\Sigma^{\frac{1}{2}}\right) \Sigma^{-\frac{1}{2}}\,\Sigma^{-\frac{1}{2}} \left(\Sigma^{\frac{1}{2}}\,\eta\right) \le R{^2} \\
    &\implies \eta^{\top}\eta \le R^{2} \\
    &\implies \eta \in \mathcal{E}(R^{2}; 0, I_{n})
\end{align}
\end{subequations}
Consequently, we have \citep[\S 3.3.2]{gleason2021lagrangian}:
\begin{subequations}
\begin{align}
    \mathbb{P}_{w}\!\left(w \in \mathcal{E}(R^{2}; \mu, \Sigma)\right) &= \mathbb{P}_{\eta}\!\left(\eta \in \mathcal{E}(R^{2}; 0, I_{n})\right) \\
    &= \mathbb{P}_{\eta}\!\left(\chi^{2}(n) \le R^{2}\right) \\
    &= F_{\chi^{2}(n)}(R^{2})
\end{align}
\end{subequations}
where, since $\eta$ is a normal random variable, $\eta^{\top}\eta = \norm{\eta}_{2}^{2}$ can be replaced with $\chi^{2}(n)$, which is a chi-squared random variable with $n$ degrees of freedom, $F_{\chi^{2}(n)}(R^{2})$ being its cumulative distribution function.

Given a user-specified probability, $p$, we set:
\begin{align}
    R^{2} = F_{\chi^{2}(n)}^{-1}(p) \label{eq:chi_inv}
\end{align}
which gives us:
\begin{align}
    F_{\chi^{2}(n)}(R^{2}) = p
\end{align}
and hence, gives us the following guarantee:
\begin{align}
    \mathbb{P}_{w}\!\left(w \in \mathcal{E}(R^{2}; \mu, \Sigma)\right) = p
\end{align}

\subsection{Robust Controllable Tube Generation} \label{subsec:robust-controllable-tube}

Robust controllable tube generation involves the contruction of bounded disturbance sets, which we present in \S\ref{subsubsec:bounded-disturbance-set-construction}, and the tightening, i.e., \emph{robustification}, of the control and state constraint sets, which we present in \S\ref{subsubsec:control-constraint-robustification} and \S\ref{subsubsec:state-constraint-robustification}, respectively. The set recursion algorithm for robust controllable tube generation is provided in \S\ref{subsubsec:state-constraint-robustification}.

\subsubsection{Bounded Disturbance Set Construction} \label{subsubsec:bounded-disturbance-set-construction}

While $\wx_{k}$ and $\wu_{k}$ in Equations \eqref{eq:uncertain_state} and \eqref{eq:uncertain_controls}, respectively, are independent across time, $\hat{w}_{k}$ is not, due to the coupling of $\wx_{k}$ between time-steps. For the purpose of bounded disturbance set construction, and, consequently, robust controllable tube generation, however, we make the assumption that $\hat{w}_{k}$ are independent, i.e., we assume that $\hat{w}_{k} \sim \N(0, \Sigma^{\hat{w}}_{k})$, where $\Sigma^{\hat{w}}_{k} \defeq \blkdiag\{\Sigma^{u}, \Sigma^{x}_{k+1}, \Sigma^{x}_{k}\} \in \mathbb{S}^{\tilde{n}_{u} +\,2\,\tilde{n}_{x}}$, $k = 1, \ldots, N-1$. Note that this assumption is conservative, since it ignores the temporal correlation in $\hat{w}_{k}$ (the temporal correlation manifests as additional affine constraints on $w_{k}$, which are ignored, leading to robustification with respect to a larger disturbance set). Further, since a Gaussian distribution is closed under linear transformations of random variables \citep[Theorem 3.3.3]{tong2012multivariate}, $w_{k}$ is also Gaussian.

The corresponding bounded disturbance sets are constructed as follows. First, we consider the following ellipsoids (see Equation \eqref{eq:ellipsoid}) for $k = 1, \ldots, N-1$:
\begin{align}
    \mathcal{E}_{k}(R_{k}^{2}; 0, \Sigma^{\hat{w}}_{k}) = \{q \in \R^{\tilde{n}_{u} +\,2\,\tilde{n}_{x}} \mid q^{\top} (\Sigma^{\hat{w}}_{k})^{-1} q \le R_{k}^{2}\}
\end{align}
and for $k = N$:
\begin{align}
    \mathcal{E}_{N}(R_{N}^{2}; 0, \Sigma^{x}_{N}) = \{q \in \R^{\tilde{n}_{x}} \mid q^{\top} (\Sigma^{x}_{N})^{-1} q \le R_{N}^{2}\}
\end{align}
Since the robust controllable tube generated using these bounded disturbance sets will be tightened with respect to the worst-case disturbances from these sets, the tube will be guaranteed to be robust with respect to the original disturbance vectors, $\hat{w}_{k} \defeq (\wu_{k}, \wx_{k+1}, \wx_{k}) \in \R^{\tilde{n}_{u} +\,2\,\tilde{n}_{x}}$, $k = 1, \ldots, N-1$, as well.

Next, to get the effective disturbance sets for $k = 1, \ldots, N-1$, each in $\R^{n_{x} + 1}$, we take linear transformations of $\mathcal{E}_{k}(R_{k}^{2}; 0, \Sigma^{\hat{w}}_{k})$ with respect to $M$, which are ellipsoids themselves \citep{kurzhanskiy2006ellipsoidal}, i.e.,
\begin{subequations}
\begin{align}
    \mathcal{W}_{k} &\defeq \{M q \mid M \in \R^{(n_{x} + 1) \times (\tilde{n}_{u} +\,2\,\tilde{n}_{x})},\ q \in \mathcal{E}_{k}(R_{k}^{2}; 0, \Sigma^{\hat{w}}_{k})\} \\
    &\phantom{:}= \{l \in \R^{n_{x} + 1} \mid l \in \mathcal{E}_{k}(R_{k}^{2}; 0, M \Sigma^{\hat{w}}_{k} M^{\top})\} \label{eq:effective-disturbance-ellipsoid}
\end{align}
\end{subequations}
For $k = N$, we have:
\begin{subequations}
\begin{align}
    \mathcal{W}_{N} &\defeq \{E_{w}^{x}\,q \mid E_{w}^{x} \in \R^{(n_{x} + 1) \times \tilde{n}_{x}},\ q \in \mathcal{E}_{N}(R_{N}^{2}; 0, \Sigma^{x}_{N})\} \\
    &\phantom{:}= \{l \in \R^{n_{x} + 1} \mid l \in \mathcal{E}_{N}(R_{N}^{2}; 0, E_{w}^{x}\,\Sigma^{x}_{N} (E_{w}^{x})^{\!\top})\} \label{eq:effective-disturbance-ellipsoid-N}
\end{align}
\end{subequations}
Further, in accordance with Equation \eqref{eq:chi_inv}, we set $R_{k}^{2} = F_{\chi^{2}(\tilde{n}_{u} +\,2\,\tilde{n}_{x})}^{-1}(p)$, $k = 1, \ldots, N-1$, and $R_{N}^{2} = F_{\chi^{2}(\tilde{n}_{x})}^{-1}(p)$, wherein, in order to guarantee that the probability of the trajectory at the final-step $N$ lying inside the terminal set $\CS_{N}$ is $\lambda$, we choose $p$ as follows \citep{gleason2021lagrangian}:
\begin{align}
    p = \lambda^{\frac{1}{N}}
\end{align}
where $\lambda \in (0, 1)$ and $N$ is the horizon length. That is to say that with this choice of $p$, the probability that the trajectory at the terminal time-step, $N$, will lie inside the terminal set, is at least $p^{N} = \lambda$.

\subsubsection{Robustification of Control Constraints} \label{subsubsec:control-constraint-robustification}

For the closed-loop system to be robust to control disturbances, we need to tighten the control constraint set such that the commanded control lying in the tightened constraint set implies that the true uncertain control lies in the original control constraint set, for all possible control disturbances within the bounded control disturbance set. Again, the tightened constraints need to be polytopic for robust controllable tube construction, but can be conic for forward rollout. Note that the robustified control constraint set needs to be considered in the forward rollout in the robust case.

Tightening of convex constraints in general is problem-dependent, since it depends on the specific structure of the constraint set under consideration; for example, a robustified halfspace constraint can be modeled as an SOC constraint \citep[\S 4.4.2]{Boyd2004}. In \S\ref{subsec:pdg_robust}, we provide one such way to achieve robustification of the convex control constraints with respect to the worst-case control disturbance for the autonomous precision landing problem.

\subsubsection{Robustification of State Constraints and Robust Set Recursion} \label{subsubsec:state-constraint-robustification}

While we explicitly robustify the control constraint set, as shown in \S\ref{subsubsec:control-constraint-robustification}, the state constraint sets are implictly robustified via the backward set-recursion for robust controllable tube generation, given by Algorithm \ref{alg:controllable_tube_recursion_robust}. Note that Algorithm \ref{alg:controllable_tube_recursion_robust} is similar to Algorithm \ref{alg:controllable_tube_recursion_deterministic}, with the key differences being (i) the use of the robustified control constraint set, $\tilde{\mathcal{U}}_{\text{robust}}$, and (ii) the Pontryagin difference step, which implicitly robustifies the state constraints with respect to the effective disturbance sets, $\mathcal{W}_{k}$, $k = 1, \ldots, N$. Note that $\mathcal{W}_{k}$ (at each time-step) accounts for the influence of both the state and control disturbances on the evolution of the state.
\begin{algorithm}[H]
\small
\caption{Set Recursion for Robust Controllable Tube Generation}\label{alg:controllable_tube_recursion_robust}
    \vspace{1em}
    \begin{flushleft}
        \textbf{Inputs:} $N$, $\Ad$, $\Bd$, $\gd$, $\tilde{\mathcal{X}}$, $\tilde{\mathcal{X}}_{\mathrm{f}}$, $\tilde{\mathcal{U}}_{\text{robust}}$, $\mathcal{W}_{[1:N]}$
    \end{flushleft}
    \begin{algorithmic}[1]
    \vspace{1em}
    \State $\CS_{N}$ $\leftarrow$ $\tilde{\mathcal{X}}_{\mathrm{f}} \ominus \mathcal{W}_{N}$ \Comment{$\ominus$: Pontryagin difference}\vspace{0.5em}
    \For{$k = (N-1), (N-2), \ldots, 2, 1$}\vspace{0.25em}
    \State $\CS_{k} = \tilde{\mathcal{X}} \cap \Ad^{-1} \left((\CS_{k+1} \ominus \mathcal{W}_{k}) \oplus \left(-\Bd\,\tilde{\mathcal{U}}_{\text{robust}} - \gd\right) \right)$ \Comment{$\oplus$: Minkowski sum}\vspace{0.25em}
    \EndFor \vspace{1em}
    \end{algorithmic}
    \begin{flushleft}
        \textbf{Return:} $\CS_{1, \ldots, N}$\vspace{1em}
    \end{flushleft}
\end{algorithm}

%% file: sections/subsections/open_loop_v_closed_loop.tex
\subsection{The Robust Optimal Control Problem} \label{subsec:robust-optimal-control}

In the deterministic case, the open-loop and closed-loop optimal control problems are equivalent \citep{bertsekas2012dynamic}. In the robust case, however, there are differences—specifically, as shown in \S\ref{subsubsec:open-loop-vs-closed-loop}, the robust open-loop optimal control problem described in \S\ref{subsubsec:robust-open-loop} is more conservative than the robust closed-loop optimal control problem described in \S\ref{subsubsec:robust-closed-loop}.

\subsubsection{Robust Open-Loop Optimal Control} \label{subsubsec:robust-open-loop}

Consider the optimal control problem given by Problem \ref{prob:dt_ocp_template_polytopic}, but with the inclusion of bounded additive disturbances, $w_{k} \in \W_{k} \subset \R^{n_{x} + 1}$, where $\W_{k}$, $k = 1, \ldots, N$, are compact; see Problem \ref{prob:dt_ocp_template_polytopic_robust_open_loop}. This problem represents the robust \emph{open-loop} optimal control problem, since any control input along the future horizon only depends on the state at the beginning of the horizon.
\ifjgcd
\else
    \vspace{\parskip}
\fi
\begin{problem}
\begin{mybox}
\begin{center}
    \underline{\textbf{Problem \ref{prob:dt_ocp_template_polytopic_robust_open_loop}: Robust Open-Loop Optimal Control Problem (Polytopic)}}
\end{center}
{\footnotesize
\begin{mini!}|s|[2]
{\substack{N,\,c_{1}\\u_{1},\,\ldots,\,u_{N-1}, \\ \sigma_{1},\, \ldots,\,\sigma_{N-1}}}{\!\!\!\!\!\max_{\substack{w_{1},\,\ldots,\,w_{N}}}\quad c_{1} \label{eq:open-loop-cost-function}}
{\label{prob:dt_ocp_template_polytopic_robust_open_loop}}{}
\addConstraint{\hphantom{\!\!\!\max_{\substack{w_{1},\,\ldots,\,w_{N}}}\quad}}{(x_{k+1}, c_{k+1}) = \Ad\cdot(x_{k}, c_{k}) + \Bd\cdot(u_{k}, \sigma_{k}) + \gd + w_{k},}{\quad k = 1, \ldots, N-1 \label{eq:dt-dynamics-augmented-robust-open-loop}}
\addConstraint{\hphantom{\!\!\!\max_{\substack{w_{1},\,\ldots,\,w_{N}}}\quad}}{(x_{k}, c_{k}) \in \tilde{\X},}{\quad k = 1, \ldots, N}
\addConstraint{\hphantom{\!\!\!\max_{\substack{w_{1},\,\ldots,\,w_{N}}}\quad}}{(u_{k}, \sigma_{k}) \in \tilde{\U},}{\quad k = 1, \ldots, N-1}
\addConstraint{\hphantom{\!\!\!\max_{\substack{w_{1},\,\ldots,\,w_{N}}}\quad}}{w_{k} \in \W_{k},}{\quad k = 1, \ldots, N}
\addConstraint{\hphantom{\!\!\!\max_{\substack{w_{1},\,\ldots,\,w_{N}}}\quad}}{x_{1} = x_{\mathrm{i}}}
\addConstraint{\hphantom{\!\!\!\max_{\substack{w_{1},\,\ldots,\,w_{N}}}\quad}}{(x_{N}, c_{N}) + w_{N} \in \tilde{\X}_{\mathrm{f}}}
\end{mini!}
}%
\end{mybox}
\end{problem}
\ifjgcd
\else
    \vspace{\parskip}
\fi
In practice, the optimal control problem is first robustified (i.e., the constraints are tightened) with respect to the worst-case disturbances from the bounded sets (owing to the maximization over them), and then the entire sequence of controls is solved for at once.

\subsubsection{Robust Closed-Loop Optimal Control} \label{subsubsec:robust-closed-loop}

Consider Problem \ref{prob:dt_ocp_template_polytopic} again, with the same bounded additive disturbances, $w_{k} \in \W_{k} \subset \R^{n_{x} + 1}$, $k = 1, \ldots, N$. Now, instead of sequential maximization over disturbance and minimization over controls, we leverage the additional information provided by the value of the current state \citep{bertsekas2012dynamic}, and formulate an interleaved minimization and maximization problem; see Problem \ref{prob:dt_ocp_template_polytopic_robust_closed_loop}. This problem represents the robust \emph{closed-loop} optimal control problem, since the control input at any given point in time explicitly depends on the state at the same time.
\ifjgcd
\else
    \vspace{\parskip}
\fi
\begin{problem}
\begin{mybox}
\begin{center}
    \underline{\textbf{Problem \ref{prob:dt_ocp_template_polytopic_robust_closed_loop}: Robust Closed-Loop Optimal Control Problem (Polytopic)}}
\end{center}
{\footnotesize
\begin{mini!}|s|[2]
{\substack{\vphantom{|}N,\,c_{1}}}{\!\!\!\!\!\min_{u_{1}}\max_{w_{1}}\ldots\min_{u_{N-1}}\max_{w_{N-1}}\max_{w_{N}}\quad c_{1}}
{\label{prob:dt_ocp_template_polytopic_robust_closed_loop}}{}
\addConstraint{\hphantom{\!\!\!\min_{u_{1}}\max_{w_{1}}\ldots\min_{u_{N-1}}\max_{w_{N-1}}\max_{w_{N}}\quad}}{(x_{k+1}, c_{k+1}) = \Ad\cdot(x_{k}, c_{k}) + \Bd\cdot(u_{k}, \sigma_{k}) + \gd + w_{k},}{\quad k = 1, \ldots, N-1 \label{eq:dt-dynamics-augmented-robust-closed-loop}}
\addConstraint{\hphantom{\!\!\!\min_{u_{1}}\max_{w_{1}}\ldots\min_{u_{N-1}}\max_{w_{N-1}}\max_{w_{N}}\quad}}{(x_{k}, c_{k}) \in \tilde{\X},}{\quad k = 1, \ldots, N}
\addConstraint{\hphantom{\!\!\!\min_{u_{1}}\max_{w_{1}}\ldots\min_{u_{N-1}}\max_{w_{N-1}}\max_{w_{N}}\quad}}{(u_{k}, \sigma_{k}) \in \tilde{\U},}{\quad k = 1, \ldots, N-1}
\addConstraint{\hphantom{\!\!\!\min_{u_{1}}\max_{w_{1}}\ldots\min_{u_{N-1}}\max_{w_{N-1}}\max_{w_{N}}\quad}}{w_{k} \in \W_{k},}{\quad k = 1, \ldots, N}
\addConstraint{\hphantom{\!\!\!\min_{u_{1}}\max_{w_{1}}\ldots\min_{u_{N-1}}\max_{w_{N-1}}\max_{w_{N}}\quad}}{x_{1} = x_{\mathrm{i}}}
\addConstraint{\hphantom{\!\!\!\min_{u_{1}}\max_{w_{1}}\ldots\min_{u_{N-1}}\max_{w_{N-1}}\max_{w_{N}}\quad}}{(x_{N}, c_{N}) + w_{N} \in \tilde{\X}_{\mathrm{f}}}
\end{mini!}
}%
\end{mybox}
\end{problem}
\ifjgcd
\else
    \vspace{\parskip}
\fi
In practice, Problem \ref{prob:dt_ocp_template_polytopic_robust_closed_loop} is solved by means of dynamic programming—first, a backward recursion, which accounts for robustification, as shown in Algorithm \ref{alg:controllable_tube_recursion_robust}, and then a forward rollout, i.e., a sequence of one-step optimal control problems, akin to Problem \ref{prob:one_step_ocp}, but with the robustified constraint sets instead.

\subsubsection{Robust Closed-Loop Optimal Control vs.\ Robust Open-Loop Optimal Control} \label{subsubsec:open-loop-vs-closed-loop}

Consider the \emph{max-min inequality} \citep[Equation (5.46)]{Boyd2004}:
\begin{align}
    \mathop{\vphantom{\inf}\sup}_{z \in Z}\, \mathop{\vphantom{\sup}\inf}_{y \in Y}\, f(y, z) \le \mathop{\vphantom{\sup}\inf}_{y \in Y}\, \mathop{\vphantom{\inf}\sup}_{z \in Z}\, f(y, z) \label{eq:max-min-inequality}
\end{align}
for any $f : \R^{n} \times \R^{m} \to \R$, any $Y \subseteq \R^{n}$, and any $Z \subseteq \R^{m}$. When the maximum and the minimum exist, Equation \eqref{eq:max-min-inequality} can be written as follows:
\begin{align}
    \max_{z \in Z}\, \min_{y \in Y}\, f(y, z) \le \min_{y \in Y}\, \max_{z \in Z}\, f(y, z) \label{eq:max-min-inequality-simplified}
\end{align}
Repeated application of the max-min inequality in Equation \eqref{eq:max-min-inequality-simplified} to Problem \ref{prob:dt_ocp_template_polytopic_robust_closed_loop} (closed-loop) results in the minimization and maximization operations grouped together, as in Problem \ref{prob:dt_ocp_template_polytopic_robust_open_loop} (open-loop). Thus, we conclude that (i) the robust closed-loop optimal control problem has a lower cost than that of its open-loop counterpart, and (ii) the open-loop approach is \emph{more conservative} than the closed-loop approach.

%% file: sections/resilient.tex
\section{Resilient Control} \label{sec:resilient}

When uncertainties/disturbances that affect the system can be modeled, controllers can be synthesized to ensure that the closed-loop system is robust to the worst-case disturbance. However, we would like for the system to remain safe even in the face of unmodeled uncertainties, i.e., \emph{unknown unknowns}, such as disruptions and faults. In other words, we want the system to be \emph{resilient}. This section deals with control in such situations, referred to as resilient control—it is a deterministic framework, designed to enable uncertainty-agnostic ``backup'' maneuvers to ensure safety.

Note that resilient control is distinct from robust control, in that decisions such as choosing the appropriate backup maneuver to execute need to be made in real-time to ensure safety with respect to unforeseen events, as opposed to characterizing the (known) sources of uncertainties that constantly (but mildly) affect the system, and synthesizing controllers offline.

Specifically, we describe two forms of resilient control in this section: (i) instantaneous reachability (\S\ref{subsec:instantaneous-reachability}), and (ii) maximal decision-deferral (\S\ref{subsec:decision-deferral}). In both cases, the controller synthesized online is cost-optimal, and hence, this approach is one of \emph{resilient optimal control}. A key benefit of our set-based framework is the ability to seamlessly incorporate robustness to modeled uncertainty into the resilient control framework—with resilience to unmodeled uncertainty—by considering robust controllable tubes.

\subsection{Instantaneous Reachability} \label{subsec:instantaneous-reachability}

In real-time control applications, it may be desirable to determine feasibility of a set of pre-selected target terminal points in a computationally tractable manner, i.e., to be able to quickly assess whether or not they are feasible (reachable), in the event that the nominal target becomes infeasible or unsafe due to an unforeseen event during the maneuver. Further, in certain applications, there could be a priori-determined safety maps corresponding to a set of targets, with a requirement that the trajectory of the system terminate in a safe region—in this case, any chosen target would have to be in the intersection of the set of safe regions (known a priori) and the physically-reachable envelope of the system (to be computed in real-time).

Consider a feasible state trajectory with an associated sequence of controls. A translation-invariant subspace (of the state space) is one wherein any translation of the entire state trajectory in that subspace is also a valid state trajectory for the same sequence of controls. 

States that lie in a translation-invariant subspace do not show up on the right-hand side of the dynamics. A good example of a translation-invariant subspace is the position space for systems that adhere to Newton's second law of motion. Translation-invariant coordinates are also known as cyclic coordinates \citep{landau1969mechanics,malyuta2022convex}.
\vspace*{0.5\baselineskip}
\begin{assumption} \label{assumption:partition}
    The state, $x$, is assumed to be partitioned as $x = (\hat{x},\,\hat{x}^{\complement})$, where $\hat{x}$ are cyclic coordinates corresponding to a translation-invariant subspace of interest, and $\hat{x}^{\complement}$ are the remaining coordinates.
\end{assumption}
Note that $\hat{x}^{\complement}$ may contain translation-invariant coordinates.

The instantaneous\footnote{We use the term ``instantaneous'' to indicate that the reachable set of terminal states at the horizon endpoint is computed with respect to the \emph{current state}, rather than solely with respect to the initial condition of the overall trajectory.} forward reachable set in a subspace of interest is the set of all terminal state coordinates in that subspace, reachable from the current state at time-index $k$, in $N-k$ steps.

Using the controllable tube, it is possible to obtain a closed-form expression for the instantaneous reachable set in any translation-invariant subspace of the state space.

Now, we are ready to present the main instantaneous reachability result:

\vspace{0.5\baselineskip}
\begin{notation}
    For a set $\mathcal{A} \subset \R^{n}$ and vector $a \in \R^{n}$, $\mathcal{A} \oplus a \defeq \mathcal{A} \oplus \{a\}$, and $-\mathcal{A} \defeq \{-z \mid z \in \mathcal{A}\}$.
\end{notation}

\vspace{0.5\baselineskip}
\begin{theorem}\label{theorem:instantaneous-reachability}
    Let Assumption \ref{assumption:partition} hold. Let $x_{k} = (\hat{x}_{k},\,\hat{x}^{\complement}_{k})$ be the current state at time-index $k \in \{1, \ldots, N\}$. Let $\CS_{k} \subset \R^{n_{x}+1}$ be the controllable set at time-index $k$ of the controllable tube, and let $\mathcal{P} \subset \R^{n_{x}}$ be the projection of $\CS_{k}$ onto the state coordinates, i.e.:
    \begin{align}
        \mathcal{P} \defeq \{x \mid (x,\,c) \in \CS_{k}\}
    \end{align}
    Assume $x_{k} \in \mathcal{P}$. Let the singleton $\{\hat{x}_{\mathrm{f}}\}$ be the target final set corresponding to the $\hat{x}$-subspace of $\mathcal{P}$. Consider the slice of $\mathcal{P}$ along $\hat{x}^{\complement}_{k}$, projected onto the $\hat{x}$-coordinates:
    \begin{align}
        \mathcal{S}(\hat{x}^{\complement}_{k}) \defeq \left\{s \mid (s,\,\hat{x}^{\complement}_{k}) \in \mathcal{P}\right\} \label{eq:slice-in-subspace}
    \end{align}
    Then, the instantaneous reachable set in the $\hat{x}$-subspace is:
    \begin{align}
        \mathcal{R}(\hat{x}_{k},\,\hat{x}^{\complement}_{k}) = \hat{x}_{k} \oplus (-\mathcal{S}(\hat{x}^{\complement}_{k})) \oplus \hat{x}_{\mathrm{f}} \label{eq:instantaneous-reachable-set}
    \end{align}
\end{theorem}
\begin{proof}
    We prove the theorem in two steps.
    
    (i) $\hat{x}_{k} \oplus (-\mathcal{S}(\hat{x}^{\complement}_{k})) \oplus \hat{x}_{\mathrm{f}} \subseteq \mathcal{R}(\hat{x}_{k},\,\hat{x}^{\complement}_{k})$
    
    Consider any $s \in \mathcal{S}(\hat{x}^{\complement}_{k})$. By the definitions of $\mathcal{P}$ and $\mathcal{S}(\hat{x}^{\complement}_{k})$, $\hat{x}_{\mathrm{f}}$ is reachable from $(s,\,\hat{x}^{\complement}_{k})$. Given the translation-invariance of the $\hat{x}$-subspace, translate the trajectory by $\tau = \hat{x}_{k} - s$. Then the translated current state is now the current state, $(s + \tau,\,\hat{x}^{\complement}_{k}) = (\hat{x}_{k},\,\hat{x}^{\complement}_{k})$, and the corresponding translated final condition in the $\hat{x}$-subspace is $\hat{x}_{\mathrm{f}} + \tau = \hat{x}_{\mathrm{f}} + \hat{x}_{k} - s$. Since $s$ is arbitrary, we have:
    \begin{align*}
        \hat{x}_{\mathrm{f}} + \hat{x}_{k} - s &\in \mathcal{R}(\hat{x}_{k},\,\hat{x}^{\complement}_{k}), \quad \forall\,s \in \mathcal{S}(\hat{x}^{\complement}_{k}) \\
        \iff \left\{\hat{x}_{\mathrm{f}} + \hat{x}_{k} - s \mid s \in \mathcal{S}(\hat{x}^{\complement}_{k}) \right\} &\subseteq \mathcal{R}(\hat{x}_{k},\,\hat{x}^{\complement}_{k}) \\
        \iff \hat{x}_{\mathrm{f}} \oplus \hat{x}_{k} \oplus (- \mathcal{S}(\hat{x}^{\complement}_{k})) &\subseteq \mathcal{R}(\hat{x}_{k},\,\hat{x}^{\complement}_{k})
    \end{align*}

    (ii) $\mathcal{R}(\hat{x}_{k},\,\hat{x}^{\complement}_{k}) \subseteq \hat{x}_{k} \oplus (-\mathcal{S}(\hat{x}^{\complement}_{k})) \oplus \hat{x}_{\mathrm{f}}$

    Consider any $r \in \mathcal{R}(\hat{x}_{k},\,\hat{x}^{\complement}_{k})$. By the definition of $\mathcal{R}(\hat{x}_{k},\,\hat{x}^{\complement}_{k})$, $r$ is reachable from $(\hat{x}_{k},\,\hat{x}^{\complement}_{k})$. Given the translation-invariance of the $\hat{x}$-subspace, translate the trajectory by $\tau = \hat{x}_{\mathrm{f}} - r$. Then the translated final condition in the subspace is now the final condition in the $\hat{x}$-subspace, $r + \tau = \hat{x}_{\mathrm{f}}$, and the corresponding translated initial condition is $(\hat{x}_{k} + \tau,\,\hat{x}^{\complement}_{k}) = (\hat{x}_{k} + \hat{x}_{\mathrm{f}} - r,\,\hat{x}^{\complement}_{k})$. By the definitions of $\mathcal{P}$ and $\mathcal{S}(\hat{x}^{\complement}_{k})$, and since the target final set in the $\hat{x}$-subspace is a singleton, we have that $\hat{x}_{\mathrm{f}}$ is reachable from $(s,\,\hat{x}^{\complement}_{k})$ if and only if $s \in \mathcal{S}(\hat{x}^{\complement}_{k})$. Now, since $r$ is arbitrary, we have:
    \begin{align*}
        \hat{x}_{k} + \hat{x}_{\mathrm{f}} - r &\in \mathcal{S}(\hat{x}^{\complement}_{k}), \quad \forall\,r \in \mathcal{R}(\hat{x}_{k},\,\hat{x}^{\complement}_{k}) \\
        \iff \left\{\hat{x}_{\mathrm{f}} + \hat{x}_{k} - r \mid r \in \mathcal{R}(\hat{x}_{k},\,\hat{x}^{\complement}_{k}) \right\} &\subseteq \mathcal{S}(\hat{x}^{\complement}_{k}) \\
        \iff \hat{x}_{\mathrm{f}} \oplus \hat{x}_{k} \oplus (- \mathcal{R}(\hat{x}_{k},\,\hat{x}^{\complement}_{k})) &\subseteq \mathcal{S}(\hat{x}^{\complement}_{k}) \\
        \iff -\mathcal{R}(\hat{x}_{k},\,\hat{x}^{\complement}_{k}) &\subseteq (-\hat{x}_{k}) \oplus \mathcal{S}(\hat{x}^{\complement}_{k}) \oplus (-\hat{x}_{\mathrm{f}}) \\
        \iff \mathcal{R}(\hat{x}_{k},\,\hat{x}^{\complement}_{k}) &\subseteq \hat{x}_{k} \oplus (-\mathcal{S}(\hat{x}^{\complement}_{k})) \oplus \hat{x}_{\mathrm{f}}
    \end{align*}
    Since both (i) and (ii) hold, $\mathcal{R}(\hat{x}_{k},\,\hat{x}^{\complement}_{k}) = \hat{x}_{k} \oplus (-\mathcal{S}(\hat{x}^{\complement}_{k})) \oplus \hat{x}_{\mathrm{f}}$.
\end{proof}

\vspace{0.5\baselineskip}
\begin{remark}
The instantaneous reachable set in the chosen translation-invariant subspace, $\mathcal{R}(\hat{x}_{k},\,\hat{x}^{\complement}_{k})$ (Equation \eqref{eq:instantaneous-reachable-set}), is nothing but a reflection of $\mathcal{S}(\hat{x}^{\complement}_{k})$ (Equation \eqref{eq:slice-in-subspace}) about the origin, translated by $(\hat{x}_{k} + \hat{x}_{\mathrm{f}})$.
\end{remark}

\vspace{0.5\baselineskip}
\begin{remark}
If the controllable set is convex and represented as a constrained zonotope, Equation \eqref{eq:instantaneous-reachable-set} can be efficiently computed in an optimization-free manner, only making use of simple set operations. This procedure is described in Algorithm \ref{alg:forward_reachable_set_cz}, and is demonstrated by means of a simple example in Figure \ref{fig:instantaneous-reach}.
\end{remark}

Once the instantaneous reachable envelope is computed, a list of pre-selected target terminal points can be assessed for feasibility by checking for containment within the envelope, i.e., if a target terminal point is contained within the envelope, it is guaranteed to be reachable; if it is not contained in the envelope, it is guaranteed to be unreachable. Also, given a pre-determined safety map corresponding to the terminal set, a target terminal point can be chosen such that it lies in the intersection of the set of safe regions and the instantaneous reachable envelope. Further, there could be situations where the goal is to move as far away from the nominal target terminal point as possible, in which case the system can be commanded to divert to an appropriately chosen point on the boundary of the envelope.

\vspace{0.5\baselineskip}
\begin{remark}
If the target final set is not a singleton and has a nonempty interior, then the corresponding $\mathcal{R}(\hat{x}_{k},\,\hat{x}^{\complement}_{k})$ is no longer the instantaneous reachable set (in the $\hat{x}$-subspace). However, it is guaranteed to contain—i.e., it is a superset/an outer-approximation of—the true instantaneous reachable set (in the $\hat{x}$-subspace). Specifically, condition (i) in the proof of Theorem \ref{theorem:instantaneous-reachability} will hold, but condition (ii) need not. However, this outer-approximation may still be useful, especially in cases where the goal is to ensure that the target final set is free of debris, for example. Further, this bears resemblance to the notion of the \emph{obstacle region} in robot motion planning \citep[\S4.3]{lavalle2006planning}.
\end{remark}

\begin{algorithm}[H]
\small
\caption{Computation of the Instantaneous Reachable Set in a Translation-Invariant Subspace}\label{alg:forward_reachable_set_cz}
    \vspace{1em}
    \begin{flushleft}
        \textbf{Inputs:} (1) Current controllable set: $\CS_{k}$ \\
        \hphantom{\textbf{Inputs:}}\!\; (2) Current states for which to compute the reachable set: $\hat{x}_{k}$ \\
        \hphantom{\textbf{Inputs:}}\!\; (3) Target final states for which to compute the reachable set: $\hat{x}_{\mathrm{f}}$ \\
        \hphantom{\textbf{Inputs:}}\!\; (4) Complement of the current initial states for which to compute the reachable set: $\hat{x}^{\complement}_{k}$
    \end{flushleft}
    \begin{algorithmic}[1]
    \vspace{1em}
    \State $\mathcal{P}$ $\gets$ project onto the state coordinates in $\CS_{k}$
    \State $\mathcal{S}$ $\leftarrow$ slice $\mathcal{P}$ along $\hat{x}_{k}^{\complement}$ and project it onto the $\hat{x}$ coordinates
    \State $\mathcal{R}$ $\leftarrow$ $\hat{x}_{k} \oplus (-\mathcal{S}) \oplus \hat{x}_{\mathrm{f}}$ \label{line:instantaneous-reachable-set} \vspace{1em}
    \end{algorithmic}
    \begin{flushleft}
        \textbf{Return:} $\mathcal{R}$\vspace{1em}
    \end{flushleft}
\end{algorithm}

\vspace{\parskip}
\begin{figure}[!htb]
    \centering
    \includegraphics[width=0.5\linewidth]{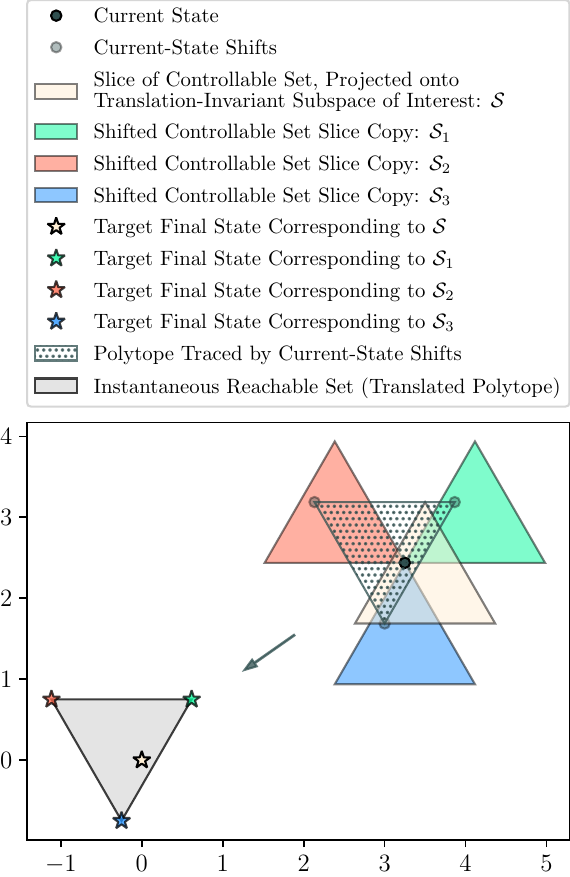}
    \caption{Computation of the instantaneous reachable set in a translation-invariant subspace of interest, given a slice of a current-state-containing controllable set projected onto that subspace.}
    \label{fig:instantaneous-reach}
\end{figure}

\subsection{Maximal Decision-Deferral} \label{subsec:decision-deferral}

In real-time trajectory generation applications, it is often the case that more information comes in (through hazard detection sensors, for instance) as the system gets closer to the nominal target. While the robust control approach we propose can account for bounded and well-characterized state and control uncertainties (see \S\ref{sec:robust}), it is not equipped to mitigate \emph{unmodeled} uncertainties/disruptions—such as a dangerous event occurring or a hazard/obstacle being detected at the last minute—that might lead to mission failure if the system were to continue proceeding towards the nominal target.

In such cases, it is desirable to have a collection of pre-selected targets as a contingency measure. While \S\ref{subsec:instantaneous-reachability} dealt with determining feasibility of these targets, in this subsection, we determine how to ensure reachability to a collection of targets for as long as possible, which in turn enables divert decision-making as late as possible.

Deferred-decision trajectory optimization (\ddto{}) was recently introduced as a deterministic open-loop framework to tackle such problems \citep{elango2022deferring, elango2025deferred}. In this work, we propose a tractable set-based, closed-loop implementation of \ddto{}, that can also seamlessly incorporate robustness to modeled uncertainties.

For this, we construct controllable tubes associated with each of the targets offline. The controllable tubes can either be deterministic, or robust to state and/or control uncertainties, as described in \S\ref{sec:robust}. In the case where the targets (final conditions) differ only in a translation-invariant subspace, it suffices to construct one controllable tube and consider translated copies of it. In the one-step OCP solved online, we require the system to stay in the \emph{intersection} of controllable sets, with a prioritized ordering of targets to determine the intersection logic.

If the chosen intersection is nonempty, the one-step OCP is solved. If the OCP returns a feasible solution, we proceed to the next intersection check. If the OCP is infeasible, we select the next best intersection or target based on the prioritized ordering, and proceed. The forward rollout algorithm with decision-deferral for the case of two targets—one nominal and one backup—is given by Algorithm \ref{alg:forward_rollout_ddto}. The nominal target is the highest-priority target.
\ifjgcd
\else
    \vspace{\parskip}
\fi
\begin{algorithm}[!htb]
\small
\caption{Forward Rollout with Decision-Deferral}\label{alg:forward_rollout_ddto}
    \vspace{1em}
    \begin{flushleft}
        \textbf{Inputs:} (1) Initial condition: $x_{\mathrm{i}}$\\[0.25em]
        \hphantom{\textbf{Inputs:}}\!\; (2) Controllable tube for the nominal target: $\CS^{\text{nominal}}_{1,\ldots,N}$\\
        \hphantom{\textbf{Inputs:}}\!\; (3) Controllable tube for the backup target: $\CS^{\text{backup}}_{1,\ldots,N}$
    \end{flushleft}
    \begin{algorithmic}[1]
    \vspace{1em}
    \State $\CS^{\text{effective}}_{[1, \ldots, N]}$ $\gets$ $\CS^{\text{nominal}}_{[1, \ldots, N]} \cap \CS^{\text{backup}}_{[1, \ldots, N]}$\vspace{0.5em} 
    \State $k^{\star}$ $\gets$ \texttt{optimal\_horizon}$(x_{\mathrm{i}}, \CS^{\text{effective}}_{[1, \ldots, N]})$\Comment{Algorithm \ref{alg:horizon}}\vspace{0.5em}
    \State $\textsc{branch}$ $\gets$ $\textsc{false}$\vspace{0.5em}
    \State $x_{k^{\star}} \gets x_{\mathrm{i}}$\vspace{0.5em}
    \For{$k = k^{\star}\!, \dots, N-1$}\vspace{0.25em}
    \If{$\textsc{branch}$ is $\textsc{false}$}\vspace{0.25em}
    \If{$\CS^{\text{effective}}_{k+1}$ is empty \textbf{OR} one-step OCP is infeasible}\vspace{0.25em}
    \State $\textsc{branch}$ $\gets$ $\textsc{true}$\vspace{0.25em}
    \State $\CS^{\text{effective}}_{[1, \ldots, N]}$ $\gets$ $\CS^{\text{nominal}}_{[1, \ldots, N]}$\vspace{0.25em}
    \State $k^{\star}$ $\gets$ \texttt{optimal\_horizon}$(x_{\mathrm{i}}, \CS^{\text{effective}}_{[1, \ldots, N]})$\Comment{Algorithm \ref{alg:horizon}}\vspace{0.25em}
    \Else\vspace{0.25em}
    \State $x_{k+1}, u_{k}$ $\gets$\vspace{0.25em} \texttt{one\_step\_optimal\_control}$(x_{k}, \CS^{\text{effective}}_{k+1})$ \Comment{solve Problem \ref{prob:one_step_ocp}}\vspace{0.25em}
    \EndIf\vspace{0.25em}
    \Else\vspace{0.25em}
    \State $x_{k+1}, u_{k}$ $\gets$\vspace{0.25em} \texttt{one\_step\_optimal\_control}$(x_{k}, \CS^{\text{effective}}_{k+1})$ \Comment{solve Problem \ref{prob:one_step_ocp}}\vspace{0.25em}
    \EndIf\vspace{0.25em}
    \EndFor\vspace{1em}
    \end{algorithmic}
    \begin{flushleft}
        \textbf{Return:} $u_{k^{\star}\!,\ldots,N-1}$\Comment{optimal control sequence}\vspace{1em}
    \end{flushleft}
\end{algorithm}
\vspace{\parskip}
In Algorithm \ref{alg:forward_rollout_ddto}, we assume that the controllable tubes associated with the two targets have the same number of controllable sets (which would be the case if the controllable tube for the backup target is a translated copy of the controllable tube for the nominal target, for instance). The containment check in the optimal horizon computation also checks for nonemptiness of the intersection.

We make use of the key insight that optimal trajectories form a tree-like structure and do not exhibit arbitrary clumping, i.e., once the trajectory \emph{trunk} (common segment) branches out, the \emph{branch} trajectories do not intersect again \citep[\S4 and Figure 3]{elango2025deferred}, and only check for emptiness until either an intersection is empty or the problem becomes infeasible—once an empty intersection or an infeasible problem is encountered, another optimal horizon computation is triggered, and the system proceeds to the nominal target (for which feasibility is guaranteed) for the remainder of the forward rollout. If the nominal target becomes unsafe while the system is in a valid intersection of controllable tubes, the system can feasibly divert to the backup target, since reachability to it was maintained by design.

Within the constraints imposed by the problem structure resulting from the chosen intersection logic, the resulting trajectory maintains reachability to the targets for as long as possible, i.e., it maintains maximal decision-deferrability, and further, since we directly leverage the optimal control approach developed in \S\ref{sec:free-final-time}, in the deterministic case, the (free-final-time) trajectory thus obtained is globally optimal with respect to the chosen cost metric.

%% file: sections/pdg.tex
\section{Autonomous Precision Landing} \label{sec:landing}

To demonstrate the controllable tube-based optimal, robust, and resilient control framework that we propose in the previous sections, we consider an autonomous precision landing case study. We formulate the optimal control problem in \S\ref{subsec:pdg_formulation}, describe the closed-loop simulation setup in \S\ref{subsec:sim_setup}, and demonstrate free-final-time optimal control in \S\ref{subsec:pdg_deterministic}, robust control in \S\ref{subsec:pdg_robust}, and resilient control in \S\ref{subsec:pdg_resilient}. We use \texttt{pycvxset} \citep{pycvxset} for all the set-based computations.

\subsection{Formulation} \label{subsec:pdg_formulation}

The original nonconvex continuous-time $3$-DoF minimum-fuel precision landing guidance optimal control problem \citep {acikmese2007convex} is given by Problem \ref{prob:ct_pdg_nonconvex}.
\begin{figure}[!htb]
\begin{problem}
\begin{mybox}
\begin{center}
    \underline{\textbf{Problem \ref{prob:ct_pdg_nonconvex}: Continuous-Time Precision Landing Problem (Nonconvex)}}
\end{center}
\begin{mini!}[2]
{\tf,\,T}{-m(\tf)} 
{\label{prob:ct_pdg_nonconvex}}{}
\addConstraint{\forall\,t \in [0, \tf]}\nonumber
\addConstraint{\dot{r}(t) = v(t)}
\addConstraint{\dot{v}(t) = \frac{1}{m(t)}\, T(t) -g\,e_{z}}
\addConstraint{\dot{m}(t) = -\alpha\,\norm{T(t)}_{2}}
\addConstraint{T_{\min} \le \norm{T(t)}_{2} \le T_{\max} \label{eq:thrust-bounds}}
\addConstraint{e_{z}^{\top}\,T(t) \ge \norm{T(t)}_{2}\, \cos\theta_{\max} \label{eq:thrust-pointing}}
\addConstraint{H_{\textsc{gs}}\,r(t) \le h_{\textsc{gs}}}
\addConstraint{\norm{r(t)}_{\infty} \le r_{\max} \label{eq:position-bounds}}
\addConstraint{\norm{v(t)}_{\infty} \le v_{\max} \label{eq:velocity-bounds}}
\addConstraint{\mdry \le m(t) \le \mwet}
\addConstraint{r(0) = r_{\mathrm{i}},\enskip v(0) = v_{\mathrm{i}},\enskip m(0) = \mwet}
\addConstraint{r(\tf) = r_{\mathrm{f}},\: v(\tf) = v_{\mathrm{f}},\: m(\tf) \ge \mdry}
\end{mini!}
\end{mybox}
\end{problem}
\end{figure}
In Problem \ref{prob:ct_pdg_nonconvex}, $\tf$ is the final time, $r(t) \in \R^{3}$ is the position, $v(t) \in \R^{3}$ is the velocity, $m(t) \in \R$ is the vehicle mass, $T(t) \in \R^{3}$ is the thrust vector in Cartesian coordinates, $g \in \R_{++}$ is the acceleration due to gravity, $\alpha$ is the thrust-specific fuel consumption, $T_{\min},\,T_{\max} \in \R_{++}$ are the minimum and maximum thrust magnitude bounds, respectively, $e_{z}^{\top} \defeq (0, 0, 1)$, $\theta_{\max}$ is the maximum tilt angle from the vertical, $H_{\textsc{gs}} \in \R^{n_{h} \times 3}$ and $h_{\textsc{gs}} \in \R^{n_{h}}$ are the parameters defining the glideslope constraint (as the intersection of $n_{h}$ halfspaces; see \citep[Equations S10 and S11]{malyuta2022convex}), $r_{\max},\,v_{\max} \in \R_{++}$ are the component-wise bounds on position and velocity, respectively, $\mwet,\,\mdry \in \R_{++}$ are the wet-mass and dry-mass of the vehicle, respectively, $r_{\mathrm{i}},\,r_{\mathrm{f}} \in \R^{3}$ are the initial and final conditions for the position, respectively, and $v_{\mathrm{i}},\,v_{\mathrm{f}} \in \R^{3}$ are the initial and final conditions for the velocity, respectively. Since this is a three-degree-of-freedom formulation, all quantities are defined relative to a reference frame fixed to the celestial body under consideration.

The autonomous precision landing problem has been extensively studied, and it has been shown that the nonconvex thrust lower bound constraint in Equation \eqref{eq:thrust-bounds} can be losslessly convexified by means of a convex relaxation, such that the solution to the relaxed problem is a globally optimal solution to the original problem \citep{acikmese2007convex, accikmecse2013lossless, malyuta2022convex}. The thrust pointing constraint given by Equation \eqref{eq:thrust-pointing} would be nonconvex for $\theta_{\max} > 90^{\circ}$, but we only require $\theta_{\max} \le 90^{\circ}$. Note that the state constraints in Equations \eqref{eq:position-bounds} and \eqref{eq:velocity-bounds} are only included to ensure that the state is bounded (as required for controllable tube generation). Typically, $r_{\max}$ and $v_{\max}$ are set to large values, however, and are not expected to be active.

Problem \ref{prob:ct_pdg_nonconvex} can be equivalently written in terms of the mass-normalized thrust by means of a log-mass transformation \citep{acikmese2007convex}, with $u(t) \defeq \frac{1}{m(t)} T(t)$ and $z(t) \defeq \ln m(t)$. Consequently, the control magnitude bounds become:
\begin{align}
    T_{\min}\,e^{-z(t)} \le \norm{u(t)}_{2} \le T_{\max}\,e^{-z(t)} \label{eq:control-bounds-log-mass}
\end{align}
and the log-mass dynamics are:
\begin{align}
    \dot{z}(t) = -\alpha\,\norm{u(t)}_{2} \label{eq:log-mass-dynamics}
\end{align}
both of which are nonconvex. First, we replace the time-varying bounds in Equation \eqref{eq:control-bounds-log-mass} with the following time-invariant bounds, such that the resulting constraint is a conservative approximation of the original constraint:
\begin{align}
    \frac{T_{\min}}{\mdry} \le \norm{u(t)}_{2} \le \frac{T_{\max}}{\mwet} \label{eq:time-invariant-control-bounds}
\end{align}
Note that the lower bound constraint in Equation \eqref{eq:time-invariant-control-bounds} is still nonconvex. For this, we can employ one of the following two remedies: 
\vspace{-0.25em}
\begin{enumerate}[label=(\roman*), itemsep=-0.25em]
    \item The standard convex relaxation to convexify the nonconvex lower bound \citep{acikmese2007convex, accikmecse2013lossless, malyuta2022convex}, i.e.,
    \begin{align}
        \nonumber \\[-1em]
        \sigma(t) \ge \frac{T_{\min}}{\mdry} \label{eq:lcvx-lower-bound-constraint}
        \\[-1em] \nonumber
    \end{align}
    Note that this does not exactly fit the template of Problem \ref{prob:ct_ocp_template_nonconvex}, since this serves as a convex relaxation to the original nonconvex constraint. However, when implemented, we observe that the relaxation is tight in practice, i.e., lossless convexification holds \citep{acikmese2007convex, accikmecse2013lossless, malyuta2022convex}. Although we consider a polytopic \emph{inner} approximation to the conic frustum defined by the control magnitude bounds, satisfaction of the polytopic constraint does not necessarily guarantee satisfaction of the original constraint—the lower bound constraint can be violated, since the polytopic approximation we consider is of a convex relaxation to a nonconvex constraint set. This has the adverse consequence of expanding the feasible control space—beyond that of the original control constraint set—in some regions. That said, the degree to which this constraint is violated can be arbitrarily controlled by choosing the number of points to be spread on the unit sphere. Further, this issue can be eliminated entirely by ensuring that the projection of the polytopic approximation onto the non-lifted space is an inner approximation to the original nonconvex set. \label{item:lcvx-case}
    \item Replacing the lower bound constraint with the following conservative convex constraint:
    \begin{align}
        \nonumber \\[-1em]
        e_{z}^{\top}u(t) \ge \frac{T_{\min}}{\mdry} \label{eq:conservative-lower-bound-constraint}
        \\[-1em] \nonumber
    \end{align}
    With the conservative constraint given by Equation \eqref{eq:conservative-lower-bound-constraint}, the control magnitude, and in turn, the slack variable, are implicitly lower-bounded by $\frac{T_{\min}}{\mdry}$. With Equation \eqref{eq:conservative-lower-bound-constraint}, the problem fits the template of Problem \ref{prob:ct_ocp_template_nonconvex}, in which case the only nonconvexity stems from the log-mass dynamics (Equation \eqref{eq:log-mass-dynamics}). For the discretized version of this template, there exists a blanket lossless convexification guarantee \citep{vinod2025set}. We consider the conservative convex constraint given by Equation \eqref{eq:conservative-lower-bound-constraint} in the remainder of this paper. \label{item:conservative-case}
\end{enumerate}

Further, we replace $\norm{u}_{2}$ by $\sigma$ in the corresponding pointing constraint, $e_{z}^{\top} u(t) \ge \norm{u(t)}_{2}\, \cos\theta_{\max}$ (even though it is convex), hence converting the second-order cone into an equivalent halfspace, which mitigates the need for an additional polytopic approximation.

\input{sections/subsections/cost_to_go}

\input{sections/subsections/soc}

\subsection{Closed-Loop Simulation Setup} \label{subsec:sim_setup}

In this subsection, we describe the closed-loop optimal control architecture in \S\ref{subsubsec:architecture} and the problem and simulation parameters in \S\ref{subsubsec:parameters}, and discuss the sources of approximation and conservatism in \S\ref{subsubsec:sources-of-approx}.

\subsubsection{Architecture} \label{subsubsec:architecture}

The closed-loop optimal control architecture for autonomous precision landing is shown in Figure \ref{fig:architecture}. Specifically, the figure depicts the architecture adopted for the deterministic simulations in this section. 

\vspace{\parskip}
\begin{figure}[!b]
    \centering
    \includegraphics[width=\linewidth]{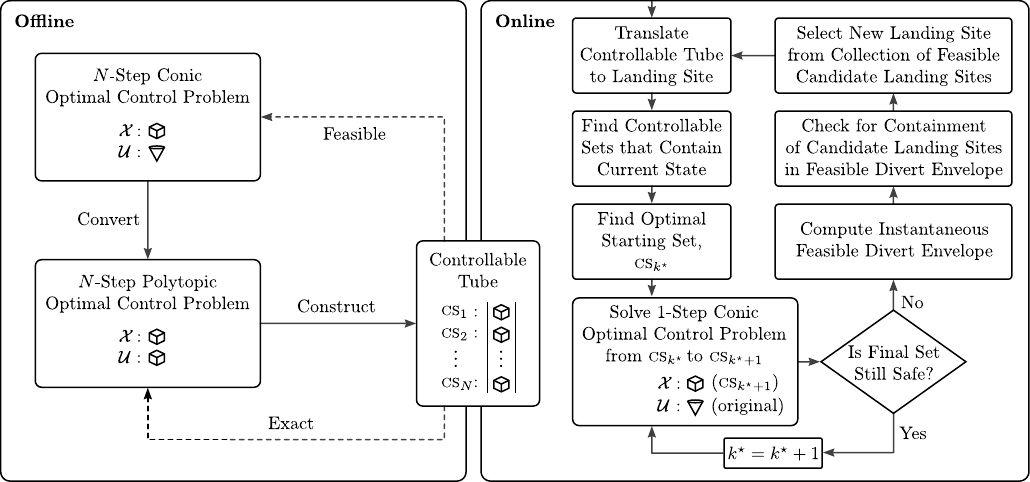}
    \vspace{-1em}
    \caption{A closed-loop optimal control architecture for autonomous precision landing with integrated divert capabilities.}
    \label{fig:architecture}
\end{figure}

The framework involves both offline and online components. Offline, we first obtain a polytopic approximation (Problem \ref{prob:dt_ocp_template_polytopic}) of the conic optimal control problem, Problem \ref{prob:dt_ocp_template_conic}. Next, we generate the controllable tube based on the polytopic optimal control problem, using Algorithm \ref{alg:controllable_tube_recursion_deterministic} (or Algorithm \ref{alg:controllable_tube_recursion_robust} in the robust case), the recursion for which is guaranteed to terminate for the autonomous precision landing problem owing to the limited fuel. The controllable tube is then stored onboard the system. What follows is the description of a notional closed-loop control framework for autonomous precision landing with integrated divert capabilities, as shown in Figure \ref{fig:architecture}.

Given the translation-invariance of the position subspace, rather than generating a controllable tube for each target landing site, we only construct one controllable tube corresponding to the origin, and either (i) translate the tube to the landing site in question via a Minkowski sum operation, or (ii) shift the reference frame such that the corresponding landing site becomes the origin. Next, we perform containment checks (see Table \ref{tab:set-operations}) to determine which of the controllable sets in the controllable tube contain the current state. Once these controllable sets have been determined, Algorithm \ref{alg:horizon} is executed to determine the optimal starting controllable set, i.e., the one that corresponds to the lowest cost-to-go value. Then, the forward rollout (Algorithm \ref{alg:forward_rollout}) is commenced from that controllable set, which involves solving a sequence of one-step optimal control problems (of the form of Problem \ref{prob:one_step_ocp}).

The forward rollout is executed in the nominal mode of operation, i.e., as described in Algorithm \ref{alg:forward_rollout}, for as long as the target landing site is safe. In case it is deemed unsafe at any point in time, computation of the instantaneous feasible divert envelope, i.e., the instantaneous reachable set in the horizontal position subspace, is triggered. Given a collection of candidate landing sites (assumed to be priority-ordered), the next best landing site can be determined by performed containment checks against the feasible divert envelope: if the landing site is contained within the divert envelope, it is guaranteed to be reachable (by construction); if it is not contained in the divert envelope, it is not reachable. It could also be used to (i) determine a feasible site in a safe region (by taking the intersection of the divert envelope with a safety map, for instance), and (ii) achieve maximal diverts or ``flyaway''s to ensure safety of the delivered payload \cite{accikmecse2014mars}, i.e., a point on the boundary of the divert envelope can be chosen in any desired direction, and a closed-loop solution to it is guaranteed to exist by construction (based on Problem \ref{prob:dt_pdg_polytopic}). Once a new landing site is chosen, the controllable tube is then translated to that site, and the process is repeated until touchdown.

\subsubsection{Problem Parameters} \label{subsubsec:parameters}

The problem parameter values chosen for the optimal, resilient, and robust control simulations are listed in Table \ref{tab:params}. In the table, $\Delta t$ refers to the discretization sampling time, and $n_{\mathrm{points}}$ refers to the number of points to be spread on the $3$D unit sphere for the polytopic approximation of the conic control constraint set; $r_{\mathrm{f}}^{\mathrm{backup}}$ is the backup landing site used in the resilient control simulations. 

The closed-loop robust control (Monte Carlo) simulation is run for a fixed final time, i.e., a fixed horizon length, $N$, in order to stay consistent with the adopted probabilistic framework. The $3$-sigma control uncertainty is denoted by $\delta_{3\sigma}^{u}$—note that $R_{u} = \frac{\delta_{3\sigma}^{u}}{3}$ and $\Sigma^{u} = R_{u}^{2}\,I_{3} \in \R^{3}$. The time-varying $3$-sigma position and velocity uncertainties are denoted by $\delta_{3\sigma}^{r}(k)$ and $\delta_{3\sigma}^{v}(k)$, respectively, where $k = 1, \ldots, N$. Similarly, note that $\Sigma_{k}^{x} = \blkdiag\!\left\{\left(\frac{\delta_{3\sigma}^{r}(k)}{3}\right)^{\!2} I_{3},\,\left(\frac{\delta_{3\sigma}^{v}(k)}{3}\right)^{\!2} I_{3}\right\} \in \R^{6}$. Finally, we note that $\delta_{3\sigma}^{r}(1)$ corresponds to a $30$ m initial $3$-sigma navigation uncertainty in position, and $\delta_{3\sigma}^{v}(1)$ corresponds to a $0.6$ m\,s$^{-1}$ initial $3$-sigma navigation uncertainty in velocity; the value for $\delta_{3\sigma}^{u}$ corresponds to roughly $0.5\%$ the maximum control magnitude.

\newsavebox{\RightColBox}
\newlength{\RightColH}
\sbox{\RightColBox}{%
  \begin{minipage}[t]{0.48\linewidth}
    \vspace*{0pt}\centering
    \begin{tabular}{>{\centering\arraybackslash}m{2.75cm} >{\centering\arraybackslash}m{4.25cm}}
    \toprule
    \textbf{Parameter} & \textbf{Value}\\
    \midrule
    $N$ & free\\
    $\Delta t$ & $3$ s\\
    $\alpha$ & $0.00115$ s\,m$^{-1}$\\
    $r_{\mathrm{i}}$ & $(875, 0, 635)$ m\\
    $r_{\mathrm{f}}^{\mathrm{backup}}$ & $(1700, 0, 0)$ m\\
    $v_{\mathrm{i}}$ & $(40, 0, -30)$ m\,s$^{-1}$\\
    $n_{\mathrm{points}}$ & $302$\\
    \midrule
    $N$ & 20\\
    $\Delta t$ & $15$ s\\
    $\alpha$ & $0.0002875$ s\,m$^{-1}$\\
    $r_{\mathrm{i}}$ & $(4000, 4000, 4000)$ m\\
    $v_{\mathrm{i}}$ & $(-10, -10, -10)$ m\,s$^{-1}$\\
    $n_{\mathrm{points}}$ & $14$\\
    $\lambda$ & $0.95$\\
    $\delta_{3\sigma}^{u}$ & $0.023$ m\,s$^{-2}$\\
    $\delta_{3\sigma}^{r}(k)$ & $1.5\,(N-k+1)$ m\\
    $\delta_{3\sigma}^{v}(k)$ & $0.03\,(N-k+1)$ m\,s$^{-1}$\\
    \bottomrule
    \end{tabular}
    \label{tab:params-robust}
  \end{minipage}%
}
\setlength{\RightColH}{\dimexpr\ht\RightColBox+\dp\RightColBox\relax}

\begin{table}[!htpb]
\centering
\begin{tabular}{@{}p{0.48\linewidth}@{\hspace{0.02\linewidth}}p{0.48\linewidth}@{}}
\begin{minipage}[t][\RightColH][t]{\linewidth}
  \vspace*{0pt}\centering
  \begin{tabular}{>{\centering\arraybackslash}m{1cm} >{\centering\arraybackslash}m{6cm}}
  \toprule
  \textbf{Parameter} & \textbf{Value}\\
  \midrule
  $g$ & $1.625$ m\,s$^{-2}$\\
  $T_{\mathrm{full}}$ & $10500$ kg\,m\,s$^{-2}$\\
  $T_{\max}$ & $0.8\,T_{\mathrm{full}}$\\
  $T_{\min}$ & $0.2\,T_{\mathrm{full}}$\\
  $\mwet$ & $1905$ kg\\
  $\mdry$ & $1505$ kg\\
  $\theta_{\max}$ & $50^{\circ}$\\
  $r_{\max}$ & $4000$ m\\
  $v_{\max}$ & $100$ m\,s$^{-1}$\\
  $r_{\mathrm{f}}$ & $(0, 0, 0)$ m\\
  $v_{\mathrm{f}}$ & $(0, 0, 0)$ m\,s$^{-1}$\\
  $\gamma_{\max}$ & $80^{\circ}$\\
  $H_{\textsc{gs}}$ & $\begin{bmatrix}\cos\gamma_{\max}&0&-\sin\gamma_{\max}\\[2pt]
  0&\cos\gamma_{\max}&-\sin\gamma_{\max}\\[2pt]
  -\cos\gamma_{\max}&0&-\sin\gamma_{\max}\\[2pt]
  0&-\cos\gamma_{\max}&-\sin\gamma_{\max}\end{bmatrix}$\\
  $h_{\textsc{gs}}$ & $(0, 0, 0, 0)$\\[0.09em]
  \bottomrule
  \end{tabular}
\end{minipage}
&
\usebox{\RightColBox}
\\
\end{tabular}
\caption{Problem parameters. Left: common; right top: closed-loop optimal and resilient control simulations; right bottom: closed-loop robust control (Monte Carlo) simulations, where $k = 1, \ldots, N$.}
\label{tab:params}
\end{table}
\vspace{-\parskip}

\subsubsection{Sources of Approximation and Conservatism} \label{subsubsec:sources-of-approx}

Here, we catalog the sources of approximation and conservatism in the proposed closed-loop control framework. As shown in Equation \eqref{eq:control-bounds-log-mass}, the original bounds on the control magnitude are (time-varying) nonlinear functions of the log-mass; instead, we only consider conservatively chosen time-invariant bounds, as shown in Equation \eqref{eq:time-invariant-control-bounds}. Further, we introduce conservatism in the control magnitude lower bound, as shown in Equation \eqref{eq:conservative-lower-bound-constraint}. Next, we consider a polytopic inner approximation to the control constraint set for controllable tube generation via set recursion, which effectively shrinks the controllable space—this approximation can be made arbitrarily accurate by choosing more points to spread on the unit sphere (see \S\ref{subsec:polytopic_cqc}).

In the robust case, specifically, there are three additional sources of conservatism: (i) the independence assumption on the effective disturbance for robustification, in which case the controllable tube is robust to the worst-case disturbance at each time-step without accounting for coupling of the state disturbances in time, (ii) invocation of the triangle and reverse triangle inequalities to tighten the control constraints in \S\ref{subsubsec:landing-control-robustification}, and (iii) the Pontryagin difference algorithm \citep[Algorithm 3]{vinod2025projection} in the set recursion for robust controllable tube generation, which returns an inner approximation of the true Pontryagin difference that worsens as the latent dimension of the constrained zonotope(s) increases, i.e., as the number of columns, $n_{g}$, of $G$ in Equation \eqref{eq:cz-definition} increases. Choosing more points on the unit sphere to obtain the polytopic approximation of the conic control set, for example, will lead to constrained zonotopes that are larger in latent dimension, thus making the Pontryagin difference more conservative. Note, however, that the conservatism introduced by (i), (ii), and (iii) is \emph{in the right direction}, i.e., they lead to the robust controllable tube being \emph{more} robust than necessary, thus preserving all the claimed probability guarantees.

\subsection{Optimal Control: The Optimal Guidance Algorithm} \label{subsec:pdg_deterministic}

We solve the deterministic, free-final-time (minimum-fuel) autonomous precision landing guidance problem using Algorithm \ref{alg:controllable_tube_recursion_deterministic} for controllable tube generation and Algorithm \ref{alg:forward_rollout} for the forward rollout. Within the forward rollout, we solve Problem \ref{prob:one_step_ocp_polytopic} as is and do not invoke Remark \ref{remark:conic_one_step_ocp} here to demonstrate global optimality. For comparison, we also solve Problem \ref{prob:dt_pdg_polytopic} directly in an open-loop setting, and show the control magnitude profiles in Figure \ref{fig:cl-v-ol}.

\begin{figure}[!htb]
    \centering
    \includegraphics[width=0.965\linewidth]{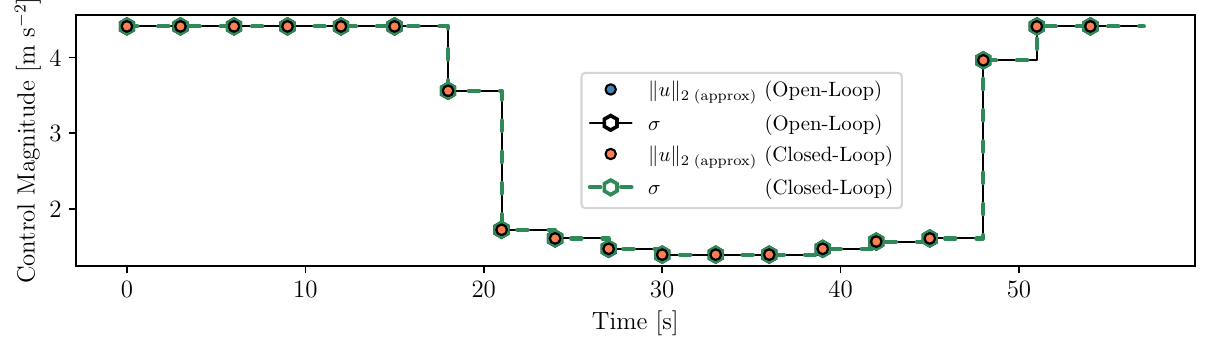}
    \caption{Algorithm \ref{alg:forward_rollout}—with the one-step optimal control problem in accordance with Problem \ref{prob:one_step_ocp_polytopic}—recovers a globally optimal solution to the polytopic open-loop \emph{free-final-time} (minimum-fuel) problem (Problem \ref{prob:dt_pdg_polytopic}), which is solved in concert with a golden section search.}
    \label{fig:cl-v-ol}
\end{figure}

%% file: sections/subsections/cost_to_go.tex
The (Mayer) cost function, after the log-mass transformation, can be given by:
\begin{align}
    J &\defeq -z(\tf)
\end{align}
which can be equivalently expressed as a running cost:
\begin{align}
    J &= -\int_{0}^{\tf} \dot{z}(\tau)\,d\tau = -\int_{0}^{\tf} -\alpha\,\sigma(\tau)\,d\tau = \int_{0}^{\tf}\alpha\,\sigma(\tau)\,d\tau
\end{align}
Now, the cost-to-go can be defined as follows:
\begin{align}
    c(t) \defeq \int_{t}^{\tf} \alpha\,\sigma(\tau)\,d\tau \label{eq:ct_cost_to_go_pdg}
\end{align}
Taking the time-derivative of Equation \eqref{eq:ct_cost_to_go_pdg}, we get:
\begin{align}
    \dot{c}(t) &= \frac{d}{dt}\int_{t}^{\tf} \alpha\,\sigma(\tau)\,d\tau = \frac{d}{dt}\int_{\tf}^{t} -\alpha\,\sigma(\tau)\,d\tau = -\alpha\,\sigma(t) \label{eq:ct_cost_to_go_derivative_pdg}
\end{align}
From Equation \eqref{eq:ct_cost_to_go_pdg}, we can derive the boundary conditions for this auxiliary dynamical system:
\begin{subequations}
\begin{align}
    c(0) &= \int_{0}^{\tf} \alpha\,\sigma(\tau)\,d\tau \label{eq:cost_to_go_initial_condition_pdg} \\
    c(\tf) &= 0 \label{eq:cost_to_go_final_condition_pdg}
\end{align}
\end{subequations}
Now, consider the following:
\begin{subequations}
\begin{align}
    z(\tf) &= z(0) - \int_{0}^{\tf} \alpha\,\sigma(\tau)\,d\tau \\
    \implies \int_{0}^{t} \alpha\,\sigma(\tau)\,d\tau &= z(0) - z(\tf) \label{eq:cost_to_go_initial_condition_integral}
\end{align}
\end{subequations}
Substituting Equation \eqref{eq:cost_to_go_initial_condition_integral} in Equation \eqref{eq:cost_to_go_initial_condition}, we get:
\begin{align}
    c(0) &= z(0) - z(\tf) = \ln\mwet - z(\tf) \label{eq:cost_to_go_initial_condition_evaluated_pdg}
\end{align}
Collecting Equations \eqref{eq:ct_cost_to_go_derivative_pdg}, \eqref{eq:cost_to_go_final_condition_pdg}, and \eqref{eq:cost_to_go_initial_condition_evaluated_pdg}, we have the following auxiliary dynamical system:
\begin{subequations}
\begin{align}
    \dot{c}(t) &= -\alpha\,\sigma(t) \\
    c(0) &= \ln\mwet - z(\tf) \\
    c(\tf) &= 0
\end{align} \label{eq:cost_to_go_dynamics_ct}
\end{subequations}
with the bounds on the cost-to-go state being:
\begin{align}
    0 \le c(t) \le \ln\mwet - \ln\mdry = \ln\frac{\mwet}{\mdry} \label{eq:cost-to-go-bounds}
\end{align}
Although the right-hand-side (RHS) of Equation \eqref{eq:cost_to_go_dynamics_ct} is identical to the RHS of the log-mass dynamics given by Equation \eqref{eq:log-mass-dynamics}, Equations \eqref{eq:cost_to_go_dynamics_ct} represent a different system due to the different boundary conditions. In this work, we employ the standard procedure of augmenting the physical system with the auxiliary cost-to-go system, making the resulting system $8$-dimensional, which, in turn, leads to $8$-dimensional polytopes (CZs) in the controllable tube.

\vspace{\baselineskip}
\begin{remark}
The similarity between the log-mass dynamics (Equation \eqref{eq:log-mass-dynamics}) and the cost-to-go dynamics (Equation \eqref{eq:cost_to_go_dynamics_ct}) can be exploited to eliminate the log-mass variable and dynamics entirely, and to equivalently rewrite the problem in terms of the cost-to-go variable and dynamics, in which case the polytopes in the controllable tube would be $7$-dimensional, instead of $8$-dimensional. Particularly, we observe (and can leverage the fact) that the log-mass and the cost-to-go are related by:
\begin{align}
    c(t) = z(t) - z(\tf)
\end{align}
with the bounds on the cost-to-go state given by Equation \eqref{eq:cost-to-go-bounds}. Using this formulation would require additional considerations to account for the initial condition of the log-mass at every step; particularly, this would require either (i) taking the intersection of the controllable set with the halfspaces defined by the log-mass bounds (prior to solving the one-step optimal control problem), in which case potential infeasibility can be detected prior to solving the one-step OCP, or (ii) imposing the halfspaces directly in the one-step OCP, in which case potential infeasibility would be detected via solving the one-step OCP itself. \label{remark:8D-to-7D}
\end{remark}

Given how well CZs scale with problem dimension, we do not observe a marked difference in performance with the formulation resulting from Remark \ref{remark:8D-to-7D}, and hence retain the $8$-dimensional formulation for simplicity. Note that the cost-to-go state, $c(t)$, serves as a proxy for the (instantaneous) ``required fuel'' to reach the target.

%% file: sections/subsections/soc.tex
Finally, we discretize the log-mass-transformed and cost-to-go-augmented problem using a ZOH control parameterization, as described in \S\ref{subsec:deterministic}—the discretized conic optimal control problem is given by Problem \ref{prob:dt_pdg_conic}.
\begin{figure}[!htb]
\begin{problem}
\begin{mybox}
\begin{center}
    \underline{\textbf{Problem \ref{prob:dt_pdg_conic}: Discrete-Time Precision Landing Problem (Conic)}}
\end{center}
{\small
\begin{mini!}[2]
{\substack{N,\\u_{1},\,\ldots,\,u_{N-1}, \\ \sigma_{1},\, \ldots,\,\sigma_{N-1}}}{c_{1}}
{\label{prob:dt_pdg_conic}}{}
\addConstraint{}{(r_{k+1}, v_{k+1}, z_{k+1}, c_{k+1}) = \Ad\cdot(r_{k}, v_{k}, z_{k}, c_{k}) + \Bd\cdot(u_{k}, \sigma_{k}) + \gd,}{\quad k = 1, \ldots, N-1}
\addConstraint{}{\norm{u_{k}}_{2} \le \sigma_{k}}{\quad k = 1, \ldots, N-1 \label{eq:dt_pdg_slack_cone}}
\addConstraint{}{\sigma_{k} \le \frac{T_{\max}}{\mwet}}{\quad k = 1, \ldots, N-1 \label{eq:dt_pdg_sigma_ub}}
\addConstraint{}{e_{z}^{\top} u_{k} \ge \frac{T_{\min}}{\mdry}}{\quad k = 1, \ldots, N-1}
\addConstraint{}{e_{z}^{\top} u_{k} \ge \sigma_{k}\, \cos\theta_{\max}}{\quad k = 1, \ldots, N-1}
\addConstraint{}{H_{\textsc{gs}}\,r_{k} \le h_{\textsc{gs}}}{\quad k = 2, \ldots, N-1}
\addConstraint{}{\norm{r_{k}}_{\infty} \le r_{\max}}{\quad k = 2, \ldots, N-1}
\addConstraint{}{\norm{v_{k}}_{\infty} \le v_{\max}}{\quad k = 2, \ldots, N-1}
\addConstraint{}{\ln\mdry \le z_{k} \le \ln\mwet}{\quad k = 2, \ldots, N-1}
\addConstraint{}{0 \le c_{k} \le \ln\frac{\mwet}{\mdry}}{\quad k = 2, \ldots, N-1}
\addConstraint{r_{1} = r_{\mathrm{i}},\enskip\!\; v_{1} = v_{\mathrm{i}},\enskip\!\; z_{1} = \ln\mwet,\enskip\!\; c_{1} \le \ln\frac{\mwet}{\mdry}}
\addConstraint{r_{N} = r_{\mathrm{f}},\: v_{N} = v_{\mathrm{f}},\: z_{N} \ge \ln\mdry,\ c_{N} = 0 }
\end{mini!}
}%
\end{mybox}
\end{problem}
\end{figure}
The state constraints in Problem \ref{prob:dt_pdg_conic} are polytopic to begin with. The control constraint set, however, is conic, due to Equation \eqref{eq:dt_pdg_slack_cone}—the control input slack constraint that arises as a result of lossless convexification is a $4$-dimensional quadratic cone (note that the following applies in both continuous time and in discrete time):
\begin{align}
    \mathcal{U}_{\text{slack}} \defeq \left\{(u,\,\sigma) \in \R^{4} \bigm| \norm{u}_{2} \le \sigma,\ \sigma \ge 0\right\} \label{eq:slack_cone}
\end{align}
where $u \in \R^{3}$ and $\sigma \in \R$. The set is closed and convex, but not compact, and a linear approximation of this set would yield an unbounded set. However, since the controllable tube generation method using constrained zonotopes requires the control input constraint set to be a bounded polytope, we also consider the upper bound constraint on the control magnitude, which is a halfspace:
\begin{align}
    \mathcal{U}_{\textsc{ub}} \defeq \left\{(u,\,\sigma) \in \R^{4} \bigm| \sigma \le \frac{T_{\max}}{\mwet}\right\} \label{eq:upper_bound}
\end{align}
Now, consider the intersection of the sets given by Equations \eqref{eq:slack_cone} and \eqref{eq:upper_bound}, $\mathcal{U}_{\text{slack}}$ and $\mathcal{U}_{\textsc{ub}}$, respectively:
\begin{align}
    \mathcal{U}_{\textsc{cqc}} \defeq \mathcal{U}_{\mathrm{slack}} \cap\, \mathcal{U}_{\textsc{ub}} = \left\{(u,\,\sigma) \in \R^{4} \bigm| \norm{u}_{2} \le \sigma,\ 0 \le \sigma \le \frac{T_{\max}}{\mwet}\right\} \label{eq:intersection}
\end{align}
The set given by Equation \eqref{eq:intersection}, $\mathcal{U}_{\textsc{cqc}}$, is a compact quadratic cone. We compute a polytopic inner approximation of $\mathcal{U}_{\textsc{cqc}}$ by following the procedure described in \S\ref{subsec:polytopic_cqc}.

With this approximation applied to the discretized precision landing guidance problem, we get the discrete-time polytopic precision landing guidance problem that fits the template of Problem \ref{prob:dt_ocp_template_polytopic}, which is shown in Problem \ref{prob:dt_pdg_polytopic}. Note that Equations \eqref{eq:dt_pdg_slack_cone} and \eqref{eq:dt_pdg_sigma_ub} in Problem \ref{prob:dt_pdg_conic} are replaced by Equation \eqref{eq:dt_pdg_cqc_approximation} in Problem \ref{prob:dt_pdg_polytopic}.
\begin{figure}[!htb]
\begin{problem}
\begin{mybox}
\begin{center}
    \underline{\textbf{Problem \ref{prob:dt_pdg_polytopic}: Discrete-Time Precision Landing Problem (Polytopic)}}
\end{center}
{\small
\begin{mini!}[2]
{\substack{N,\\u_{1},\,\ldots,\,u_{N-1}, \\ \sigma_{1},\, \ldots,\,\sigma_{N-1}}}{c_{1}}
{\label{prob:dt_pdg_polytopic}}{}
\addConstraint{}{(r_{k+1}, v_{k+1}, z_{k+1}, c_{k+1}) = \Ad\cdot(r_{k}, v_{k}, z_{k}, c_{k}) + \Bd\cdot(u_{k}, \sigma_{k}) + \gd,}{\quad k = 1, \ldots, N-1}
\addConstraint{}{(u_{k}, \sigma_{k}) \in \mathcal{U}_{\textsc{cqc}}}{\quad k = 1, \ldots, N-1 \label{eq:dt_pdg_cqc_approximation}}
\addConstraint{}{e_{z}^{\top} u_{k} \ge \frac{T_{\min}}{\mdry}}{\quad k = 1, \ldots, N-1}
\addConstraint{}{e_{z}^{\top} u_{k} \ge \sigma_{k}\, \cos\theta_{\max}}{\quad k = 1, \ldots, N-1}
\addConstraint{}{H_{\textsc{gs}}\,r_{k} \le h_{\textsc{gs}}}{\quad k = 2, \ldots, N-1}
\addConstraint{}{\norm{r_{k}}_{\infty} \le r_{\max}}{\quad k = 2, \ldots, N-1}
\addConstraint{}{\norm{v_{k}}_{\infty} \le v_{\max}}{\quad k = 2, \ldots, N-1}
\addConstraint{}{\ln\mdry \le z_{k} \le \ln\mwet}{\quad k = 2, \ldots, N-1}
\addConstraint{}{0 \le c_{k} \le \ln\frac{\mwet}{\mdry}}{\quad k = 2, \ldots, N-1}
\addConstraint{r_{1} = r_{\mathrm{i}},\enskip\!\; v_{1} = v_{\mathrm{i}},\enskip\!\; z_{1} = \ln\mwet,\enskip\!\; c_{1} \le \ln\frac{\mwet}{\mdry}}
\addConstraint{r_{N} = r_{\mathrm{f}},\: v_{N} = v_{\mathrm{f}},\: z_{N} \ge \ln\mdry,\ c_{N} = 0 }
\end{mini!}
}%
\end{mybox}
\end{problem}
\end{figure}

%% file: sections/robustification.tex
\subsection{Robust Control: Robustness to Navigation and Actuation Uncertainties} \label{subsec:pdg_robust}

In this subsection, we describe robust control in the context of the autonomous precision landing problem. Specifically, we model navigation and actuation uncertainty in \S\ref{subsubsec:nav-and-act-uncertainty}, describe practical aspects of robust controllable tube generation in \S\ref{subsubsec:landing-robust-controllable-tube}, and describe robustification of the control and state constraints in \S\ref{subsubsec:landing-control-robustification} and \S\ref{subsubsec:landing-state-robustification}, respectively. Finally, we present Monte Carlo simulation results in \S\ref{subsubsec:monte-carlo}.

\subsubsection{Navigation and Actuation Uncertainty Models} \label{subsubsec:nav-and-act-uncertainty}

We assume that the estimated position and velocity are subject to additive navigation (state) uncertainty and that the control input is subject to additive actuation (control) uncertainty, both of which are modeled as Gaussian disturbances, in accordance with \S\ref{subsec:uncertainty-modeling}. Note that we treat the augmented control slack variable, $\sigma$, such that it is not directly perturbed; instead, we robustify the original control constraints first and then convexify them (i.e., introduce $\sigma$). We also assume that the log-mass state and the cost-to-go state can be estimated perfectly (since they are functions of the \emph{applied} control input and not the commanded control input).

The bounded disturbance sets for the uncertainties, as required by Algorithm \ref{alg:controllable_tube_recursion_robust}, are constructed in accordance with \S\ref{subsec:stochastic}.

\subsubsection{Robust Controllable Tube Generation} \label{subsubsec:landing-robust-controllable-tube}

We schedule the state covariances such that the corresponding disturbance ellipses decrease in size over time (the standard deviations decrease linearly)—this accounts for the improvement in the accuracy of the navigation estimates as the system approaches the landing site (as more information is acquired by the sensors). In the robust controllable tube set recursion algorithm (Algorithm \ref{alg:controllable_tube_recursion_robust}), the terminal set constrained zonotope needs to be full-dimensional and have a \textsc{MinRow} representation for the Pontryagin difference step to be computationally tractable \citep[Definition 3 and Proposition 1]{vinod2025projection}. To generate this full-dimensional terminal set, we first perform a deterministic set recursion (Algorithm \ref{alg:controllable_tube_recursion_deterministic}) for one $\Delta t$-second time interval, and then proceed with robust set recursion. This has the added benefit of guaranteeing the existence of a solution to the original terminal set for any true state lying in the thus-generated full-dimensional terminal set, and can thus account for reserve fuel, for example, in an exact manner, without requiring conservative estimates. In practice, we observe that the conditions required for Pontryagin differencing to be tractable \citep[Definition 3 and Proposition 1]{vinod2025projection} do not hold when the set is generated via a one-step recursion. To mitigate this, we perform a two-step recursion, but with a sampling time of $\frac{\Delta t}{2}$ instead (such that the effective interval is still $\Delta t$).

\subsubsection{Robustification of Control Constraints} \label{subsubsec:landing-control-robustification}

Here, we present one approach to robustifying the control constraint set with respect to the worst-case control disturbance.

\paragraph{Control Magnitude Bounds}

To robustify the control magnitude bound constraint, we first consider the original constraint given by Equation \eqref{eq:time-invariant-control-bounds}, and tighten the bounds with respect to the worst-case disturbance from the control disturbance set. Once the tightened bounds are obtained, the lower bound constraint is replaced with the conservative convex counterpart of Equation \eqref{eq:conservative-lower-bound-constraint}.

The control magnitude constraint, Equation \eqref{eq:time-invariant-control-bounds}, with actuation uncertainty, is given by:
\begin{align}
    \frac{T_{\min}}{\mdry} &\le \|u_{k} + \wu_{k}\|_{2} \le \frac{T_{\max}}{\mwet}, \quad k = 1, \dots, N-1 \label{eq:magnitude_bounds}
\end{align}
To be able to account for the slack variable that will be introduced for lossless convexification, we seek a representation of Equation \eqref{eq:magnitude_bounds} that is solely in terms of $\norm{u_{k}}_{2}$. As such, we approximate the thrust magnitude constraint given by Equation \eqref{eq:magnitude_bounds} as follows.

Invoking the triangle and reverse triangle inequalities, we get:
\begin{align}
    \abs{\|u_{k}\|_{2} - \|\wu_{k}\|_{2}} \le \|u_{k} + \wu_{k}\|_{2} \le \|u_{k}\|_{2} + \|\wu_{k}\|_{2} \label{eq:triangle_inequality}
\end{align}
We make the assumption that $\|\wu_{k}\|_{2} \le \|u_{k}\|_{2}$, i.e., the magnitude of the bounded disturbance will be less than or equal to the control input magnitude, to ensure that any constraint tightening we perform will not make the control magnitude constraint set infeasible. With this assumption, Equation \eqref{eq:triangle_inequality} can be written as:
\begin{align}
    \|u_{k}\|_{2} - \|\wu_{k}\|_{2} \le \|u_{k} + \wu_{k}\|_{2} \le \|u_{k}\|_{2} + \|\wu_{k}\|_{2} \label{eq:triangle_inequality_w_assumption}
\end{align}
Now, our goal is to represent Equation \eqref{eq:triangle_inequality_w_assumption} such that it is solely in terms of $u_{k}$. For this, we consider the Equation \eqref{eq:triangle_inequality_w_assumption} for all possible disturbances in the disturbance set:
\begin{subequations}
\begin{align}
    \|u_{k}\|_{2} - \|\wu_{k}\|_{2} &\le \|u_{k} + \wu_{k}\|_{2} \le \|u_{k}\|_{2} + \|\wu_{k}\|_{2} \quad \forall\,\wu_{k} \in \Wu \label{eq:triangle_inequality_forall_w} \\
    \iff \inf_{\wu_{k}\,\in\,\Wu}\!\left(\|u_{k}\|_{2} - \|\wu_{k}\|_{2}\right) &\le \|u_{k} + \wu_{k}\|_{2} \le \sup_{\wu_{k}\,\in\,\Wu}\!\left(\|u_{k}\|_{2} + \|\wu_{k}\|_{2}\right) \\
    \iff \|u_{k}\|_{2} - \sup_{\wu_{k}\,\in\,\Wu}\!\|\wu_{k}\|_{2} &\le \|u_{k} + \wu_{k}\|_{2} \le \|u_{k}\|_{2} + \sup_{\wu_{k}\,\in\,\Wu}\!\|\wu_{k}\|_{2}
\end{align}
\end{subequations}
Therefore,
\begin{align}
    \|u_{k}\|_{2} - \psiu \le \|u_{k} + \wu_{k}\|_{2} \le \|u_{k}\|_{2} + \psiu \label{eq:triangle_inequality_approximation}
\end{align}
where,
\begin{align}
    \psiu \defeq \sup_{\wu_{k}\,\in\,\Wu}\|\wu_{k}\|_{2}
\end{align}
Now, the control bound constraint can be rewritten (approximated) in terms of $u_{k}$ as follows:
\begin{subequations}
\begin{align}
    \frac{T_{\min}}{\mdry} \le \|u_{k}\|_{2} - \psiu &\le \|u_{k} + \wu_{k}\|_{2} \le \|u_{k}\|_{2} + \psiu \le \frac{T_{\max}}{\mwet} \\
    \implies \frac{T_{\min}}{\mdry} + \psiu &\le \|u_{k}\|_{2} \le \frac{T_{\max}}{\mwet} - \psiu \label{eq:thrust_bound_tightened}
\end{align}
\end{subequations}
Equation \eqref{eq:thrust_bound_tightened} represents a conservative tightening of the thrust magnitude bound constraint set, in that it is guaranteed to be smaller than or equal to the tightened constraint set that would have resulted had the approximation not been introduced.

Further, we derive a closed-form expression for $\psiu$ by making the assumption that the control disturbance is isotropic, in that the projection of the effective disturbance ellipsoid (given by Equation \eqref{eq:effective-disturbance-ellipsoid}) onto the control dimensions is a $3$-dimensional ball, which in turn implies that $\Sigma_{u}$ in Equation \eqref{eq:uncertain_controls} is of the form $R_{u}^{2}\,I_{3}$, where $R_{u} > 0$ is the radius of the $3$D ball. Therefore:
\begin{align}
    \psiu &= R_{u} \\
    \implies \frac{T_{\min}}{\mdry} + R_{u} &\le \|u_{k}\|_{2} \le \frac{T_{\max}}{\mwet} - R_{u} \label{eq:tightened-magnitude-bounds}
\end{align}
Note that the control disturbance can be ellipsoidal as well if desired, but we make the isotropic assumption for simplicity. Finally, we (i) retain the upper bound constraint in Equation \eqref{eq:tightened-magnitude-bounds}, but replace the $2$-norm term with the slack variable, and (ii) replace the lower bound constraint with the conservative convex constraint (see Equation \eqref{eq:conservative-lower-bound-constraint}), to yield the following tightened constraints:
\begin{subequations}
\begin{align}
    \sigma_{k} &\le \frac{T_{\max}}{\mwet} - R_{u} \\
    e_{z}^{\top}u_{k} &\ge \frac{T_{\min}}{\mdry} + R_{u}
\end{align}
\end{subequations}

\paragraph{Thrust Pointing}

As with the control magnitude bound constraints, the thrust pointing constraint with actuation uncertainty, for $k = 1, \dots, N-1$, is robustified and reformulated as follows:
\begin{subequations}
\begin{align}
    e_{z}^{\top}(u_{k} + \wu_{k}) &\ge \|u_{k} + \wu_{k}\|_{2}\,\cos\theta_{\max}, \quad \forall\,\wu_{k} \in \Wu \\
    \iff \inf_{\wu_{k}\,\in\,\Wu}\! e_{z}^{\top}(u_{k} + \wu_{k}) &\ge \sup_{\wu_{k}\,\in\,\Wu}\! \|u_{k} + \wu_{k}\|_{2}\,\cos\theta_{\max} \\
    \impliedby \inf_{\wu_{k}\,\in\,\Wu}\! (e_{z}^{\top}u_{k} + e_{z}^{\top}\wu_{k}) &\ge \sup_{\wu_{k}\,\in\,\Wu}\! (\|u_{k}\|_{2} + \|\wu_{k}\|_{2})\,\cos\theta_{\max} \\
    \iff e_{z}^{\top}u_{k} + \inf_{\wu_{k}\,\in\,\Wu}\! e_{z}^{\top}\wu_{k} &\ge (\|u_{k}\|_{2} + \sup_{\wu_{k}\,\in\,\Wu}\! \|\wu_{k}\|_{2})\,\cos\theta_{\max} \\
    \iff e_{z}^{\top}u_{k} - \psiu &\ge (\|u_{k}\|_{2} + \psiu)\,\cos\theta_{\max} \\
    \iff e_{z}^{\top}u_{k} - R_{u} &\ge (\|u_{k}\|_{2} + R_{u})\,\cos\theta_{\max} \\
    \iff e_{z}^{\top}u_{k} - R_{u} &\ge (\sigma_{k} + R_{u})\,\cos\theta_{\max}
    \label{eq:pointing-with-uncertainty}
\end{align}
\end{subequations}

\subsubsection{Robustification of State Constraints} \label{subsubsec:landing-state-robustification}

The position and velocity states are assumed to be subject to time-varying uncertainty. While the log-mass and cost-to-go states are assumed to be estimated perfectly, and while we do not explicitly subject the augmented control slack variable, $\sigma$, to uncertainty, the influence of actuation uncertainty on them calls for additional consideration.

Specifically, since there are no explicit sources of disturbance affecting the log-mass and cost-to-go states, the Pontryagin difference step in the robust set recursion algorithm (Algorithm \ref{alg:controllable_tube_recursion_robust}) is agnostic to the influence of control disturbances on them. As such, we need to ``manually'' ensure that the bounds on these states are appropriately tightened.

The log-mass dynamics with actuation uncertainty are given as follows, for $k = 1, \ldots, N-1$:
\begin{align}
    z_{k+1} = z_{k} - \alpha\, \|u_{k} + \wu_{k}\|_{2} \label{eq:uncertain_mass_dynamics}
\end{align}
We consider the worst-case mass depletion and for this, we consider the most adversarial disturbance. From the RHS of Equation \eqref{eq:triangle_inequality_approximation}, since $\alpha > 0$, we get:
\begin{align}
    z_{k+1} &= z_{k} - \alpha \left(\|u_{k}\|_{2} + \psiu\right) = z_{k} - \alpha \left(\|u_{k}\|_{2} + R_{u}\right) \label{eq:worst-case-log-mass}
\end{align}
Similarly, the worst-case cost-to-go ``depletion'', for $k = 1, \ldots, N-1$, would be:
\begin{align}
    c_{k+1} = c_{k} - \alpha \left(\|u_{k}\|_{2} + R_{u}\right) \label{eq:worst-case-cost-to-go}
\end{align}
Rather than explicitly tightening the bounds for these states—which would result in time-varying state constraint sets—we directly consider the worst-case depletion dynamics given by Equations \eqref{eq:worst-case-log-mass} and \eqref{eq:worst-case-cost-to-go} in the robust set recursion (with the original bounds)—this is an equivalent approach that implicitly achieves the effect of tightening the bounds, since both $z$ and $c$ are strictly monotonically decreasing quantities.

The Pontryagin difference step in Algorithm \ref{alg:controllable_tube_recursion_robust} accounts for robustification of the constraints on the remaining states that are influenced by control and/or state disturbances.

\subsubsection{Monte Carlo Simulation} \label{subsubsec:monte-carlo}

To demonstrate the closed-loop robust control framework and highlight the probabilistic guarantees, we perform a Monte Carlo simulation with the problem parameters given in Table \ref{tab:params}. For the same problem parameters (and boundary conditions), we run $100$ cases with different random disturbance sequences sampled from Gaussian distributions, in accordance with \S\ref{sec:robust}. We set the probability that the terminal state estimate, $\tilde{y}_{N}$, lies in the robustified terminal set, $\tilde{\mathcal{X}}_{\mathrm{f}} \ominus \mathcal{W}_{N}$—which corresponds to the probability that the true terminal state, $y_{N}$, lies in the true terminal set, $\tilde{\mathcal{X}}_{\mathrm{f}}$—to be $\lambda = 95\%$. In other words, the true terminal state is guaranteed to lie in the true terminal set with an $95\%$ probability. We plot the terminal values of the true position in Figure \ref{fig:robust-position} and the true velocity in Figure \ref{fig:robust-velocity}; all $100$ trials are successful, in that all the true terminal state values lie inside the true terminal set in all $100$ trials. Note that the success-rate is $100\%$ despite the requested probability only being $95\%$, owing to the conservatism in the robustification (in the right direction).
\vspace{\baselineskip}
\begin{figure}[!htb]
    \centering
    \includegraphics[width=\linewidth]{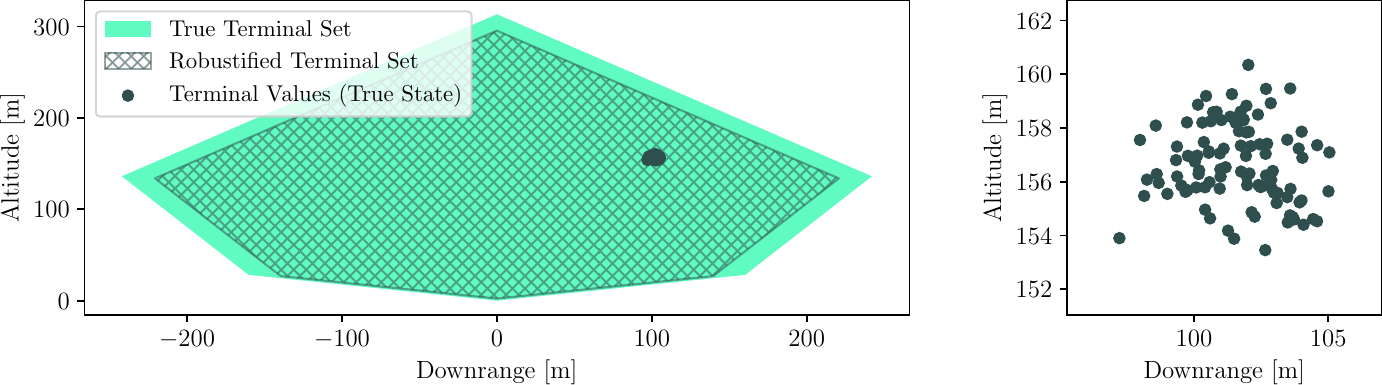}
    \caption{Terminal position values from the Monte Carlo simulation, projected onto the $x$-$z$ plane.}
    \label{fig:robust-position}
\end{figure}
\begin{figure}[!htb]
    \centering
    \includegraphics[width=\linewidth]{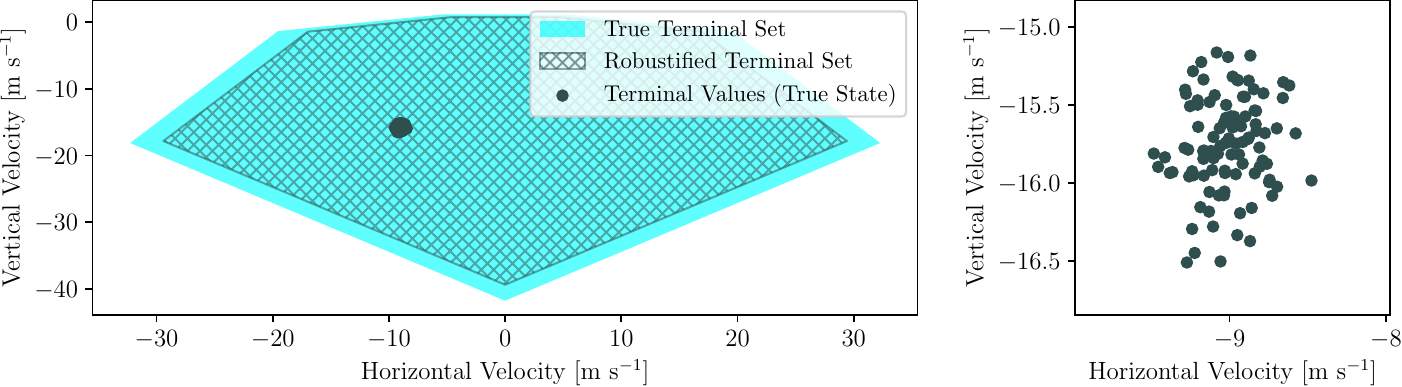}
    \caption{Terminal velocity values from the Monte Carlo simulation, projected onto the $x$-$z$ plane.}
    \label{fig:robust-velocity}
\end{figure}

%% file: sections/resilience.tex
\subsection{Resilient Control: Instantaneous Divert Envelope and Decision-Deferral} \label{subsec:pdg_resilient}

We demonstrate two modes of closed-loop resilient control for autonomous precision landing, the first based on the computation of the instantaneous feasible divert envelope, i.e., the instantaneous reachable set in the horizontal position subspace (see \S\ref{subsec:instantaneous-reachability}), and the second based on the set-based deferred-decision trajectory optimization framework (see \S\ref{subsec:decision-deferral}).

\begin{figure}[H]
    \centering
    \begin{mybox2}
    \includegraphics[width=0.825\linewidth]{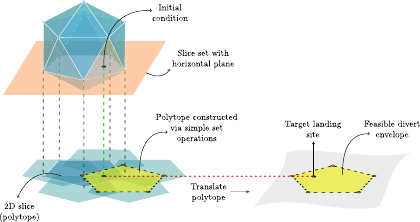}
    \end{mybox2}
    \caption{Construction of the instantaneous feasible divert envelope for Problem \ref{prob:dt_pdg_polytopic} using Algorithm \ref{alg:forward_reachable_set_cz}. The $3$D set depicted is the projection of the $8$D controllable set onto the position coordinates.}
    \label{fig:divert-envelope}
\end{figure}

The first mode of resilient control relies on computation of the instantaneous feasible divert envelope (based on Algorithm \ref{alg:forward_reachable_set_cz}), and is depicted in Figure \ref{fig:divert-envelope}. We note that divert envelope thus generated is \emph{exact} with respect to Problem \ref{prob:dt_pdg_polytopic}. Here, we consider two landing sites, a nominal site and a divert (backup) landing site—the numerical values for parameters chosen for these simulations are provided in Table \ref{tab:params}. In Figure \ref{fig:no-ddto}, we show the (independent) optimal free-final-time trajectories corresponding to the two landing sites, the corresponding control magnitude profiles (in which lossless convexification holds), the instantaneous reachable set (i.e., the instantaneous divert envelope) at the $1$st time-step (corresponding to time $t = 0$ s), and the instantaneous reachable set at the $8$th time-step (corresponding to time $t = 21$ s). Note that at $t = 0$ s, both the target landing sites are reachable, but at $t = 21$ s, neither of the reachable sets contains the other, i.e., by committing to a landing site too early (at $t = 0$ s) and performing the forward rollout (Algorithm \ref{alg:forward_rollout}) accordingly, the other landing site is no longer reachable at $t = 21$ s. 

The second mode of resilient control involves the set-based counterpart of the \emph{deferred-decision trajectory optimization} (\ddto{}) framework \citep{elango2025deferred}. Here, we consider the same scenario as the previous case, but instead, use the forward rollout algorithm with decision-deferral (Algorithm \ref{alg:forward_rollout_ddto}). In this mode, the system is required to remain in the intersection of the controllable tube (corresponding to the nominal site) and the translated controllable tube (corresponding to the divert site) for as long as possible. The two trajectories, the corresponding control magnitude profiles (in which lossless convexification holds), and the instantaneous reachable sets at $t = 0$ s and $t = 21$ s, are shown in Figure \ref{fig:ddto}. Note that both the target landing sites are reachable at (and until) $t = 21$ s in this case, unlike in the previous divert case, thus demonstrating the benefit of deferring decision. In the event one of the target landing sites becomes unsafe during the first $21$ seconds, with decision-deferral, we preserve the ability to divert to the other site, whereas this would not be possible in the feasible divert case.

\begin{figure}[!htpb]
    \centering
    \includegraphics[width=\linewidth]{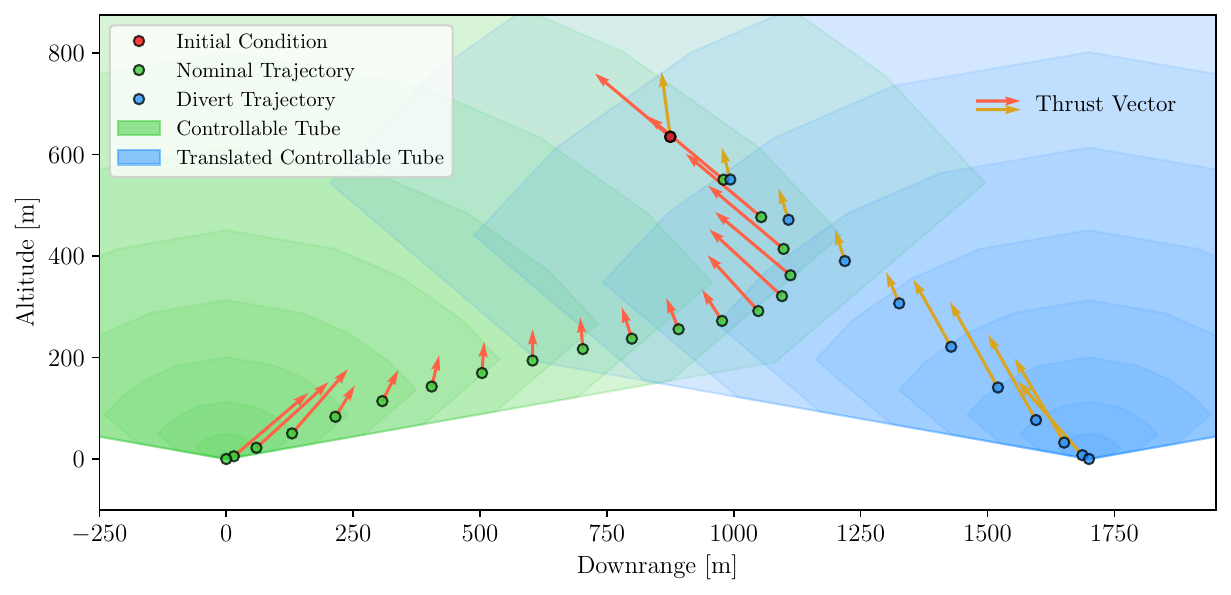}
    \hspace*{0.75em}\includegraphics[width=0.98\linewidth]{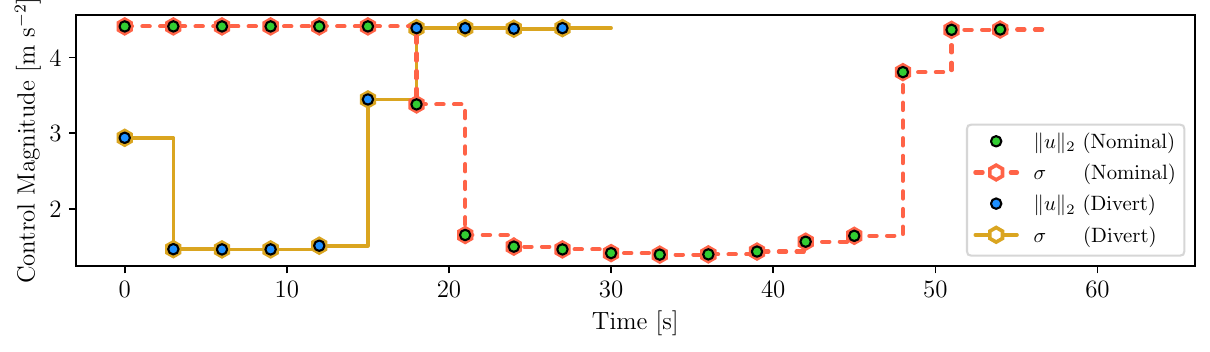}
    \includegraphics[width=0.725\linewidth]{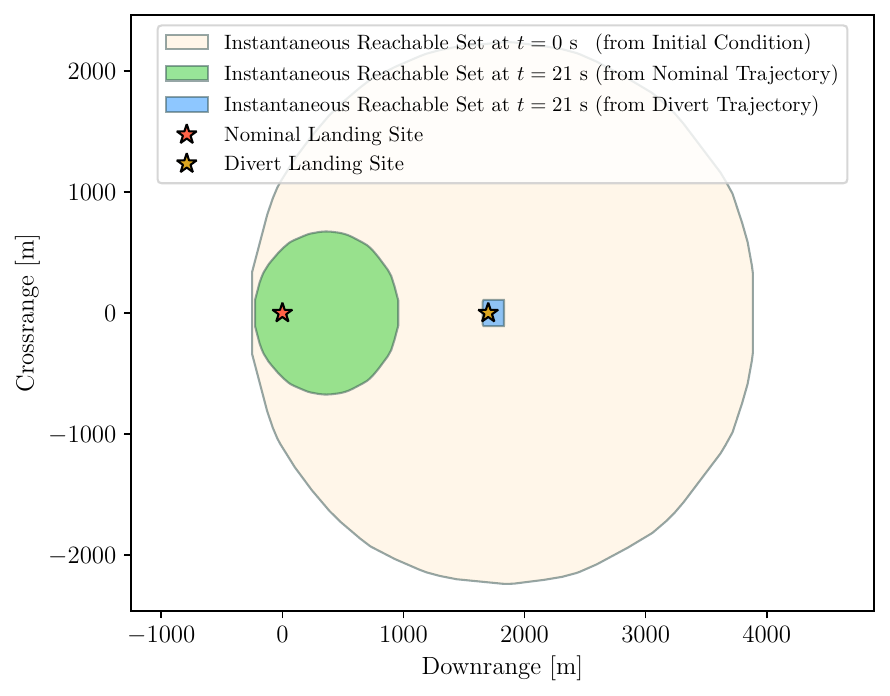}
    \caption{\textbf{Feasible divert mode}: Free-final-time trajectory optimization using the controllable tube and its translation, with Algorithm \ref{alg:forward_rollout}: at $t = 21$ s, the two divert envelopes are mutually exclusive, and thus the system en route to one landing site is unable to reach the other.}
    \label{fig:no-ddto}
\end{figure}

\begin{figure}[!htpb]
    \centering
    \includegraphics[width=\linewidth]{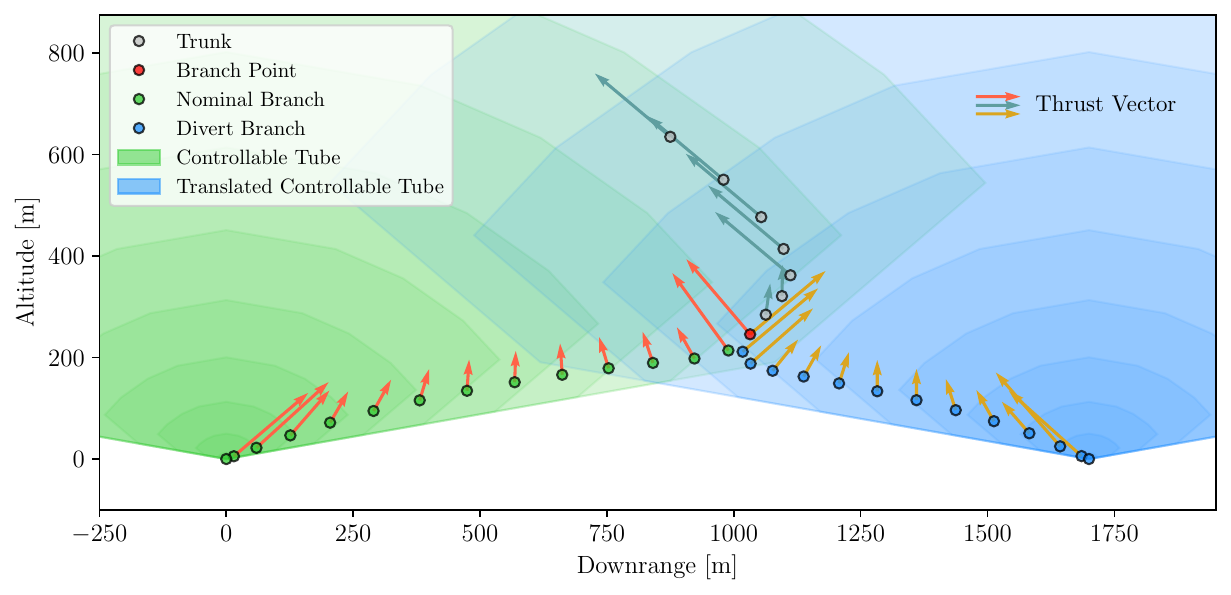}
    \hspace*{0.75em}\includegraphics[width=0.98\linewidth]{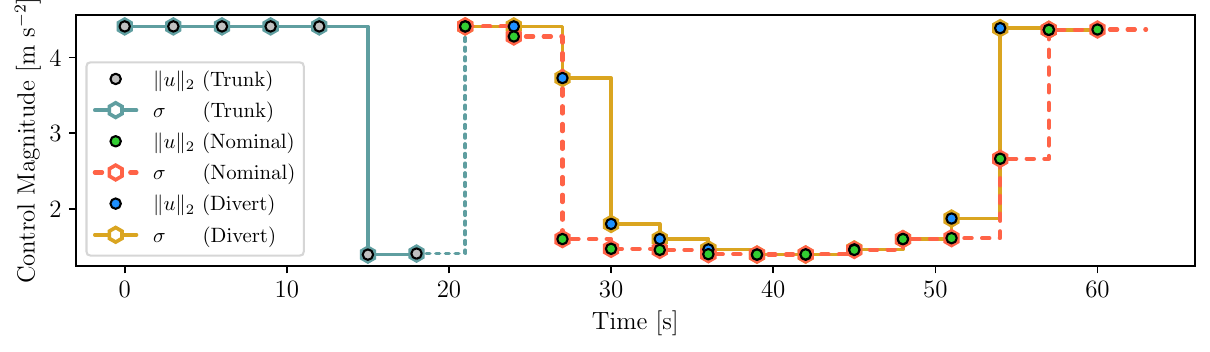}
    \includegraphics[width=0.725\linewidth]{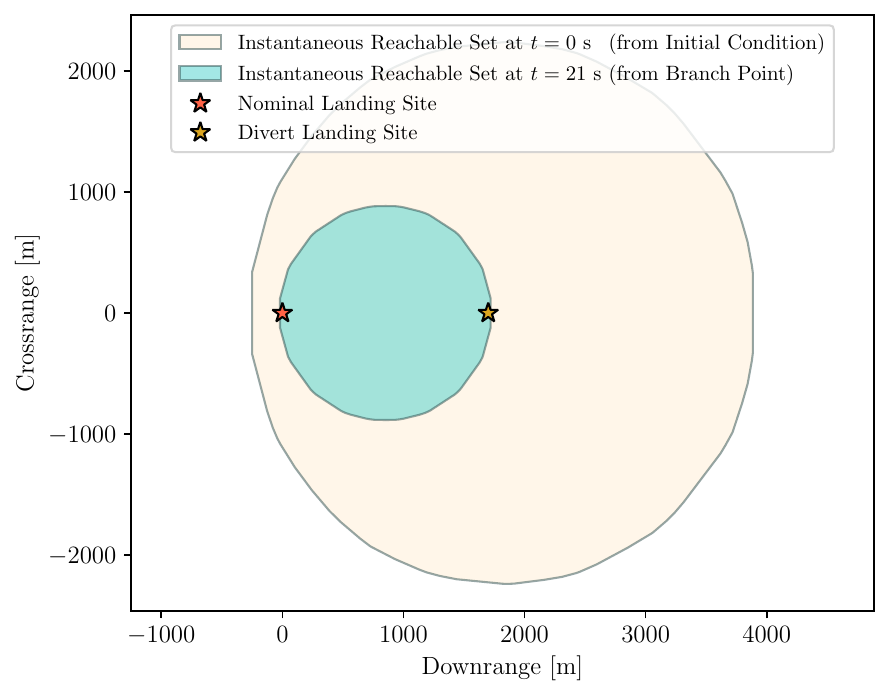}
    \caption{\textbf{Maximal decision-deferral mode}: Free-final-time \emph{deferred-decision} trajectory optimization (\ddto{}) using the controllable tube and its translation, with Algorithm \ref{alg:forward_rollout_ddto}: at $t = 21$ s, the divert envelope(s) contain both the landing sites as a result of deferring decision for as long as possible, and hence, both the sites are reachable until then, in contrast to the feasible divert mode (see Figure \ref{fig:no-ddto}).}
    \label{fig:ddto}
\end{figure}

%% file: sections/conclusions.tex
\section*{Conclusions}

We have developed a unified computational framework for closed-loop optimal, robust, and resilient predictive control for autonomous systems. Motivated by the autonomous precision landing problem, we have proposed a set-based framework for free-final-time optimal control, robust control with respect to both state and control uncertainty, instantaneous reachable set computation, and maximal decision-deferral. The framework is based on dynamic programming over controllable tubes, achieving computational tractability using constrained zonotopes, a parameterization of compact convex polytopes that enables efficient high-dimensional set operations. The proposed control architecture has been validated through a comprehensive autonomous precision landing case study.

This work aims at bridging the gap between traditional open-loop guidance approaches and closed-loop control methods for applications such as precision landing.

\ifjgcd
\else
    Avenues for future research include generalizing the class of problems considered, reducing the number of sources of approximation, and implementing the computational framework via high-performance code to assess onboard viability for autonomous systems in terms of memory requirements and solve-times.
\fi

\section*{Acknowledgment}

The authors gratefully acknowledge Kento Tomita for his feedback on an initial draft and \behcet{} for his generous guidance, insightful comments, and valuable discussions throughout the development of this work.

%% file: preamble/bibliography.tex
\begin{spacing}{1}
    \bibliography{references}
\end{spacing}
\ifjgcd
\else
    \addcontentsline{toc}{section}{References}
    \newpage
    \let\clearpage\relax
\fi